\documentclass[10pt]{article}

\usepackage{natbib}                     %
\usepackage{amsmath, amssymb, amsthm}   %
\usepackage{graphicx}                   %
\usepackage{subcaption}                 %
\usepackage{fullpage}                   %
\usepackage{hyperref}                   %
\usepackage{enumitem}                   %
\usepackage{bm}                         %
\usepackage{authblk}                    %

\usepackage{algorithm}
\usepackage{algorithmic}
\usepackage{booktabs}
\usepackage{xcolor}                     %
\makeatletter
\newcommand{\red}[1]{%
  \ifmmode
    \textcolor{red}{\mathpalette\mathRed{#1}}%
  \else
    \textcolor{red}{#1}%
  \fi
}
\newcommand{\mathRed}[2]{#2}
\makeatother

\definecolor{darkblue}{rgb}{0.0, 0.0, 0.5}  %

\hypersetup{
    colorlinks=true,
    linkcolor=darkblue,
    anchorcolor=darkblue,
    citecolor=darkblue,
    filecolor=darkblue,
    menucolor=darkblue,
    urlcolor=darkblue
}

\usepackage[nameinlink,noabbrev]{cleveref}
\numberwithin{equation}{section}

\usepackage{caption}
\newtheorem{theorem}{Theorem}
\newtheorem{lemma}{Lemma}
\newtheorem{proposition}{Proposition}
\newtheorem{corollary}{Corollary}
\theoremstyle{definition}
\newtheorem{definition}{Definition}

\newtheorem{remark}{Remark}
\newtheorem{assumption}{Assumption}

\newcommand{\rbb}{\mathbb{R}}  %
\newcommand{\ebb}{\mathbb{E}}  %
\DeclareMathOperator*{\argmin}{arg min}

\title{Generalized Stochastic Gradient Descent with Momentum Methods for Smooth Optimization}
\author[1]{Zimeng Wang}
\author[2]{Alp Yurtsever}
\affil[1,2]{Department of Mathematics and Mathematical Statistics, Umeå University, Sweden}
\affil[1]{\texttt{zimeng.wang@umu.se}}
\affil[2]{\texttt{alp.yurtsever@umu.se}}
\date{}

\begin{document}

\maketitle

\begin{abstract}
Stochastic gradient descent with momentum (SGDM) methods have become fundamental optimization tools in machine learning, combining the computational efficiency of stochastic gradients with the acceleration benefits of momentum. Despite their widespread use in practice, the theoretical understanding of SGDM remains incomplete, with most existing analyses focusing on specific momentum schemes or requiring restrictive assumptions.
In this paper, we introduce a generalized SGDM framework that unifies a broad class of momentum-based methods, including SGD with Polyak's momentum, SGD with Nesterov's momentum, and many others. We provide comprehensive convergence analyses for both convex and nonconvex optimization problems under mild smoothness and bounded variance assumptions. For convex problems, we establish general ergodic convergence results with constant parameters and derive improved iterate convergence rates with time-varying parameters. 
For nonconvex problems, we prove sublinear convergence to stationary points and establish linear convergence to a neighborhood of the optimum under the Polyak--\L{}ojasiewicz condition. Notably, our analysis allows flexible parameter choices and thus provides convergence guarantees for many existing momentum methods as special cases.
\end{abstract}

\section{Introduction}
We consider the following smooth optimization problem:
\begin{equation}\label{eq:main-prob}
    \min_{x\in\rbb^d} f(x):=\ebb_{\xi\in \mathcal{D}}[\ell(x;\xi)].
\end{equation}
Here, $\xi\in \mathcal{Z}$ denotes a data sample drawn from a distribution $\mathcal{D}$ over the sample space $\mathcal{Z}$, and $\ell(x; \xi): \rbb^d \times \mathcal{Z} \rightarrow \rbb$ is a loss function that is differentiable with respect to its first argument. This formulation is standard in machine learning applications, where $x\in \rbb^d$ represents the model parameters and $\ell(\cdot; \xi)$ measures the loss incurred by the model on data sample $\xi$. In particular, when $\mathcal{D}$ is the uniform distribution over a finite dataset $\{\xi_i\}_{i=1}^n$, the objective function $f$ in \eqref{eq:main-prob} reduces to the empirical risk $\frac{1}{n}\sum_{i=1}^n\ell(\cdot;\xi_i)$.

Gradient-based methods are fundamental tools for solving problem \eqref{eq:main-prob} in machine learning and scientific computing. The classical gradient descent (GD) method iteratively updates the iterate via $x_{k+1} = x_k - \alpha \nabla f(x_k)$, where $\alpha > 0$ is the step size. For $L$-smooth convex functions, GD with step size $\alpha = 1/L$ achieves a convergence rate of $O(1/t)$ after $t$ iterations \citep{nesterov2013introductory}. However, this rate is known to be suboptimal, as the lower bound for first-order methods in smooth convex optimization is $O(1/t^2)$ \citep{nemirovsky1983problem}.

To achieve further acceleration, momentum methods have been introduced in the optimization literature. A pioneering example is the heavy-ball (HB) method \citep{polyak1964some}, which incorporates a momentum term to accumulate information from past gradients. Although originally proposed for quadratic objectives, the HB method has since been shown to converge for general smooth convex functions \citep{ghadimi2015global}.
Subsequently, \citet{nesterov1983method} proposed a different momentum scheme by evaluating the gradient at a look-ahead point, leading to Nesterov's accelerated gradient (NAG) method. This method achieves the optimal convergence rate of $O(1/t^2)$ for smooth convex functions. These momentum methods have since been extensively studied and extended to a wide range of settings, including composite optimization \citep{beck2009fast, nesterov2013gradient}, nonconvex problems \citep{ochs2014ipiano}, and coordinate descent methods \citep{nesterov2012efficiency}.

These deterministic methods are effective when the full gradient $\nabla f$ is readily accessible. However, in many modern machine learning applications, computing the full gradient becomes prohibitively expensive as the training dataset grows large. Stochastic gradient descent (SGD) \citep{robbins1951stochastic} addresses this challenge by replacing the full gradient with a stochastic gradient estimator based on a single sample or a small mini-batch, thereby substantially reducing the per-iteration computational cost. Owing to its efficiency and scalability, SGD has become the workhorse optimization algorithm in deep learning \citep{bottou2018optimization,lecun2015deep,nguyen2019new,zhang2004solving}.

The success of momentum methods in deterministic optimization naturally motivated their extension to stochastic settings. Stochastic gradient descent with momentum (SGDM) aims to combine the computational efficiency of SGD with the acceleration benefits of momentum. In addition, the momentum mechanism helps traverse flat regions of the loss landscape and can potentially escape poor local minima \citep{ochs2015ipiasco}. Consequently, SGDM and its variants have become fundamental tools in deep learning optimization \citep{kingma2015adam,qian1999momentum,sutskever2013importance,yan2018unified} and are widely used for training deep neural networks \citep{he2016deep,simonyan2014very}.

Despite its widespread use in practice, the theoretical understanding of SGDM remains less complete than that of its deterministic counterparts. Many existing convergence results for SGDM focus on specific momentum schemes.
In particular, for SGD with Polyak’s (heavy-ball) momentum, \citet{loizou2020momentum} established global linear convergence rates for convex quadratic objectives under appropriate parameter choices. For general smooth convex objectives, \citet{sebbouh2021almost} proved almost sure convergence of the last iterate of SGD with Polyak’s momentum under time-varying step sizes and momentum parameters. Under constant parameters, \citet{liu2020improved} showed that SGD with Polyak’s momentum achieves convergence guarantees comparable to those of vanilla SGD for both nonconvex and strongly convex problems.
For stochastic composite optimization problems involving both smooth and nonsmooth components, \citet{lan2012optimal} introduced an accelerated stochastic gradient method based on Nesterov’s momentum and proved its optimality for convex objectives. This approach was later extended to nonconvex problems in \citet{ghadimi2016accelerated}.

Beyond these algorithm-specific analyses, several works have sought to develop unified frameworks that generalize classical momentum methods. For example, \citet{yan2018unified,yang2016unified} provided a unified analysis of the stochastic unified momentum (SUM) method, which encompasses both SGD with Polyak’s momentum and SGD with Nesterov’s momentum for nonconvex smooth objectives, under a bounded gradient assumption. These results were further extended to distributed settings with linear speedup guarantees in \citet{yu2019linear}. 
Additionally, \citet{ma2018quasi} proposed the quasi-hyperbolic momentum (QHM) method, which interpolates between SGD and SGD with Polyak’s momentum. When the initialization is sufficiently close to an optimal solution, \citet{gitman2019understanding} showed that QHM converges linearly for smooth and strongly convex functions. In the context of over-parameterized learning, \citet{liu2018accelerating} developed a generalized momentum method, namely the momentum-added stochastic solver (MASS), by introducing an additional compensation term to SGD with Nesterov’s momentum, and proved linear convergence for quadratic loss functions.

While the aforementioned frameworks generalize classical momentum methods, most of their convergence guarantees are established for convex or strongly convex objectives and are often restricted to special settings such as quadratic losses or interpolation regimes~\citep{liu2018accelerating}. For nonconvex problems, these works typically rely on restrictive assumptions, such as bounded gradients \citep{yan2018unified,yang2016unified}, or analyze Polyak’s momentum and Nesterov’s momentum separately to establish convergence \citep{yu2019linear}. Moreover, their analyses typically focus on constant parameter choices, and the resulting convergence rates generally match those of vanilla SGD under comparable assumptions. As a result, these results do not capture the accelerated behavior of Nesterov’s method observed in deterministic settings.
Therefore, despite substantial progress, a unified theoretical framework that systematically explains how momentum interacts with stochastic gradient methods across both convex and nonconvex regimes remains lacking.

In this work, we aim to fill this gap by providing a comprehensive convergence analysis for a generalized SGDM method that encompasses a broad class of momentum schemes for both convex and nonconvex optimization problems. Our main contributions are summarized as follows:
\begin{itemize}
    \item[1).] We introduce a generalized SGDM framework characterized by momentum parameters $\{\beta_k\}_{k\ge 1}$ and step size parameters $\{\gamma_k,\eta_k\}_{k\ge 1}$. We show that this formulation encompasses many existing momentum-based methods as special cases, including SGD with Polyak’s momentum, SGD with Nesterov’s momentum, SUM, QHM, and MASS.
    
    \item[2).] For convex problems, we establish general ergodic convergence results for the generalized SGDM with constant parameters, covering a wide range of step sizes $\gamma,\eta$ and any momentum parameter $\beta\in[0,1)$. Furthermore, through the introduction of novel Lyapunov functions built on intermediate iterates, we derive improved iterate convergence results for the generalized SGDM with time-varying parameters. With carefully chosen parameters, we recover the optimal $O(1/t^2)$ convergence rate in deterministic settings and an $O(1/t^2+\sigma/\sqrt{t})$ rate for unbiased stochastic gradients with variance bounded by $\sigma^2$.

    \item[3).] For nonconvex problems, we prove sublinear convergence to stationary points for the generalized SGDM with constant parameters. Notably, the result holds for any $\beta\in[0,1)$ and step sizes $\gamma,\eta$ whose sum is properly upper bounded. Furthermore, under an additional Polyak--\L{}ojasiewicz condition, we introduce a new sequence of Lyapunov functions to establish an improved convergence bound in terms of the objective gap, showing linear convergence to a neighborhood of the optimum.

    \item[4).] For both convex and nonconvex objectives, our analysis allows for flexible parameter choices and relies only on mild smoothness and bounded variance assumptions, without requiring bounded gradients or other restrictive conditions. To the best of our knowledge, this is the first work to provide rigorous convergence guarantees for such a general class of momentum schemes in both convex and nonconvex settings under comparable assumptions. As a result, new convergence guarantees for many existing momentum methods follow directly as special cases of our analysis.
\end{itemize}

The remainder of this paper is organized as follows. Section~\ref{sec:preliminary} introduces preliminary notations and definitions. Section~\ref{sec:algorithm} presents the generalized SGDM method and shows how it unifies several classical momentum methods. Sections~\ref{sec:convergence-cvx} and~\ref{sec:convergence-nc} provide convergence analyses for convex and nonconvex problems, respectively. Section~\ref{sec:experiments} reports numerical experiments, and Section~\ref{sec:conclusion} concludes the paper.

\section{Preliminary}
\label{sec:preliminary}
In this section, we summarize some notations and definitions that are used throughout the paper.
For a vector $v:=(v_1,\dots, v_d)^\top\in \rbb^d$, let $\|v\|$ represents its Euclidean norm, i.e., $\|v\|:=\big(\sum_{i=1}^d v_i^2\big)^{1/2}$. For $n\in\mathbb{N}_+$, we denote $[n]:=\{1,\ldots,n\}$. 

Let ($\Omega, \mathcal{F}, \mathrm{P})$ be the underlying probability space. We denote by $\left\{x_k\right\}_{k \geq 0}\subset\rbb^d$ the sequence of iterates generated by a stochastic algorithm, driven by a sequence of independent noise variables $\left\{\xi_k\right\}_{k \geq 1}$. Let $\left\{\mathcal{F}_k\right\}_{k \geq 0}$ be the natural filtration of this process, i.e., $\mathcal{F}_k= \sigma\left(x_0, \xi_1, \xi_2, \ldots, \xi_k\right)$ represents the $\sigma$-algebra generated by $x_0$ and $\left\{\xi_i\right\}_{i=1}^k$. We denote the conditional expectation given $\mathcal{F}_{k-1}$ by $\ebb_k[\cdot]:=\ebb\left[\cdot \mid \mathcal{F}_{k-1}\right]$, which represents the expectation given all information available at step $k$.
We recall the definitions of convexity, Lipschitz continuity, and smoothness as follows.
\begin{definition}[Convexity, Lipschitz continuity, and Smoothness]
Let the function $f: \rbb^d \rightarrow \rbb$ be continuously differentiable. Let $\mu\geq0, G>0$, and $L>0$.
\begin{enumerate}
  \item[1).] We say that $f$ is $\mu$-strongly convex if
  \[
  f(y) \ge f(x) + \langle\nabla f(x), y-x\rangle+\frac{\mu}{2}\|y-x\|^2,~\forall x, y\in\rbb^d.
  \]
  When the above inequality holds with $\mu=0$, we say that $f$ is convex.
  \item[2).] We say that $f$ is $G$-Lipschitz continuous if
  $
  |f(y)-f(x)|\leq G\|y-x\|,~\forall x, y\in\rbb^d.
  $
  \item[3).] We say that $f$ is $L$-smooth if its gradient is $L$-Lipschitz continuous, i.e.,
  $
  \|\nabla f(y)-\nabla f(x)\|\leq L\|y-x\|,~\forall x, y\in\rbb^d.
  $
\end{enumerate}
\end{definition}
A useful property of an $L$-smooth function $f$ is that
\begin{equation*}
    f(y) \leq f(x) + \langle\nabla f(x), y-x\rangle + \frac{L}{2} \|y-x\|^2,~\forall x,y \in \mathbb{R}^d.
\end{equation*}
Moreover, if $f$ is convex, then it additionally satisfies
\begin{equation}\label{eq: cvx-L}
	f(y) \ge f(x) + \langle\nabla f(x), y-x\rangle +\frac{1}{2L}\|\nabla f(y)-\nabla f(x)\|^2,~\forall x,y \in \mathbb{R}^d.
\end{equation}
Both inequalities are standard results in convex optimization; see \citet{nesterov2003introductory} for detailed proofs.
For a convex function $f:\rbb^d \rightarrow \rbb$ that attains a global minimizer, we denote by $x^*\in\argmin_{x\in\rbb^d} f(x)$ a minimizer of $f$ and $f^*:=f(x^*)$.

\section{Generalized SGDM Method}
\label{sec:algorithm}
We now introduce our generalized stochastic gradient descent with momentum (SGDM) framework. Let $x_1\in\rbb^d$ be an initial point and set $m_0=0$. We denote by $g_k:=\nabla \ell(x_k;\xi_{i_k})$ the stochastic gradient at iteration $k$, where $i_k$ is drawn uniformly from the index set $[n]$. The generalized SGDM updates the iterates according to
\begin{equation}\label{alg:g-sgdm}
\tag{G-SGDM}
    \begin{aligned}
        &m_k=\beta_k m_{k-1}+(1-\beta_k)g_k, \\
        &x_{k+1}=x_k-\gamma_k g_k-\eta_k m_k,
    \end{aligned}
\end{equation}
where $\beta_k\in[0,1)$ is the momentum parameter, while $\gamma_k\ge0$ and $\eta_k\in\rbb$ serve as step size parameters. Note that we allow $\eta_k$ to take negative values in this general formulation since its negativity can be compensated by a sufficiently large positive $\gamma_k$. In our analysis, we shall introduce suitably defined ``effective" step sizes that depend on the parameters $\{\beta_k,\gamma_k,\eta_k\}_{k\ge 1}$, which are required to be strictly positive.

The formulation \eqref{alg:g-sgdm} is highly general and encompasses several important special cases. When $\beta_k\equiv 0$, \eqref{alg:g-sgdm} reduces to the vanilla SGD:
\begin{equation}\label{alg:sgd}
\tag{SGD}
    x_{k+1}=x_k-(\gamma_k+\eta_k) g_k.
\end{equation}
Moreover, it includes SGD with Polyak’s (heavy-ball) momentum, SGD with Nesterov's momentum, the stochastic unified momentum (SUM) \citep{yan2018unified}, the quasi-hyperbolic momentum (QHM) method \citep{ma2018quasi}, and the momentum-added stochastic solver (MASS) \citep{liu2018accelerating} as special cases, as shown in the following proposition. Note that we assume $\eta_k\neq 0$ without loss of generality; otherwise, \eqref{alg:g-sgdm} reduces to the SGD update.
\begin{proposition}\label{prop:hb&nes&sum}
   The iterative scheme \eqref{alg:g-sgdm} encompasses the following momentum-based methods by specifying the parameters $\{\beta_k,\gamma_k,\eta_k\}_{k\ge 1}$ accordingly.
   \begin{enumerate}
    \item[(\romannumeral1).] If $\gamma_k\equiv0$, then \eqref{alg:g-sgdm} recovers the SGD with Polyak's (heavy-ball) momentum method
       \begin{equation}\label{alg:sgd-hb}
        \tag{SGD-HB}
        x_{k+1}=x_k + \frac{\beta_k \eta_k}{\eta_{k-1}} (x_k - x_{k-1}) - (1-\beta_k)\eta_k g_k.
        \end{equation}
    \item[(\romannumeral2).] If $\gamma_k=(1-\beta_k)\eta_{k-1}/\beta_k$, then \eqref{alg:g-sgdm} yields the SGD with Nesterov's momentum method:
        \begin{equation}\label{alg:sgd-nesterov}
        \tag{SGD-N}
            \begin{aligned}
                &y_{k+1} =x_k-\gamma_k g_k, \\
                &x_{k+1} =y_{k+1}+\frac{\beta_k\eta_k}{\eta_{k-1}}(y_{k+1}-y_k).
            \end{aligned}
        \end{equation}
    \item[(\romannumeral3).] If $\beta_k\equiv\beta\in[0,1)$, $\gamma_k\equiv s\alpha$, and $\eta_k\equiv \alpha/(1-\beta)-s\alpha$, where $s\in[0,1/(1-\beta)]$ and $\alpha>0$, then \eqref{alg:g-sgdm} reduces to the SUM method
        \begin{equation}\label{alg:sum}
        \tag{SUM}
            \begin{aligned}
                &y_{k+1} =x_k-\alpha g_k, \\
                &z_{k+1} = x_k - s \alpha g_k, \\
                &x_{k+1} =y_{k+1}+\beta(z_{k+1}-z_k).
            \end{aligned}
        \end{equation}
    \item[(\romannumeral4).] If $\beta_k\equiv\beta\in[0,1)$, $\gamma_k\equiv \alpha(1-\nu)$, and $\eta_k\equiv\alpha\nu$, where $\alpha>0$ and $\nu\in[0,1]$, then \eqref{alg:g-sgdm} leads to the QHM method
        \begin{equation}\label{alg:qhm}\tag{QHM}
        \begin{aligned}
        & m_k = \beta m_{k-1}+(1-\beta) g_k, \\
        & x_{k+1} = x_k-\alpha[(1-\nu) g_k+\nu m_k].
        \end{aligned}
        \end{equation}
    \item[(\romannumeral5).] If $\beta_k\equiv\beta\in[0,1)$, $\gamma_k\equiv \alpha$, and $\eta_k\equiv (\beta\alpha-\lambda)/(1-\beta)$, where $\alpha, \lambda>0$, then \eqref{alg:g-sgdm} gives the MASS method
        \begin{equation}\label{alg:mass}\tag{MASS}
            \begin{aligned}
            & y_{k+1} = x_k-\alpha g_k, \\
            & x_{k+1} = (1+\beta) y_{k+1}-\beta y_k+\lambda g_k.
            \end{aligned}
            \end{equation}
       \end{enumerate}
\end{proposition}

\begin{proof}[Proof of Proposition \ref{prop:hb&nes&sum}]
    Firstly, it is straightforward to verify (\romannumeral4) by plugging the corresponding parameters in \eqref{alg:g-sgdm}. To establish the remaining properties, we multiply the first equation in \eqref{alg:g-sgdm} by $\eta_k$ to obtain
    \begin{equation}\label{eq:sgd-n-1}
        \eta_k m_k=\frac{\beta_k \eta_k}{\eta_{k-1}} \eta_{k-1} m_{k-1}+(1-\beta_k)\eta_k g_k.
    \end{equation}
    From the second equation in \eqref{alg:g-sgdm}, we have
    \begin{align*}
        &\eta_k m_k = - x_{k+1} + x_k - \gamma_k g_k,\\
        & \eta_{k-1} m_{k-1} = - x_k + x_{k-1} - \gamma_{k-1} g_{k-1},
    \end{align*}
    applying which on \eqref{eq:sgd-n-1} gives
    \begin{equation}\label{eq:sgdm-0}
        - x_{k+1} + x_k - \gamma_k g_k = \frac{\beta_k \eta_k}{\eta_{k-1}}(- x_k + x_{k-1} - \gamma_{k-1} g_{k-1}) + (1-\beta_k)\eta_k g_k.  
    \end{equation}
    Rearranging the terms in \eqref{eq:sgdm-0} and denoting $y_{k+1}=x_k - \gamma_k g_k$, we obtain
    \begin{equation}\label{eq:sgdm}
        \begin{aligned}
           x_{k+1} &= x_k - \gamma_k g_k - \frac{\beta_k \eta_k}{\eta_{k-1}}(x_{k-1} - \gamma_{k-1} g_{k-1}) + \frac{\beta_k \eta_k}{\eta_{k-1}}x_k - (1-\beta_k)\eta_k g_k\\
           & = y_{k+1} - \frac{\beta_k \eta_k}{\eta_{k-1}}y_k + \frac{\beta_k \eta_k}{\eta_{k-1}} y_{k+1} + \Big(\frac{\beta_k \eta_k}{\eta_{k-1}}\gamma_k-(1-\beta_k)\eta_k\Big) g_k\\
           & = y_{k+1} + \frac{\beta_k \eta_k}{\eta_{k-1}} (y_{k+1} - y_k) + \Big(\frac{\beta_k \eta_k}{\eta_{k-1}}\gamma_k-(1-\beta_k)\eta_k\Big) g_k.
        \end{aligned}
    \end{equation}
    
    We now establish properties (\romannumeral1), (\romannumeral2), and (\romannumeral5) using the formulation \eqref{eq:sgdm}. For property (\romannumeral1), if $\gamma_k\equiv 0$, then by definition $y_{k+1}=x_k$, and thus \eqref{eq:sgdm} reduces to \eqref{alg:sgd-hb}. 
    For property (\romannumeral2), if $\gamma_k=(1-\beta_k)\eta_{k-1}/\beta_k$, then the last term in \eqref{eq:sgdm} vanishes, and we obtain
    \[
    x_{k+1} = y_{k+1} + \frac{\beta_k \eta_k}{\eta_{k-1}} (y_{k+1} - y_k),
    \]
    which is precisely \eqref{alg:sgd-nesterov}. For property (\romannumeral5), substituting $\beta_k\equiv\beta\in[0,1)$, $\gamma_k\equiv \alpha$, and $\eta_k\equiv (\beta\alpha-\lambda)/(1-\beta)$ into \eqref{eq:sgdm} yields
    \[
    x_{k+1} =y_{k+1} + \beta(y_{k+1} - y_k) + \lambda g_k=(1+\beta)y_{k+1} - \beta y_k + \lambda g_k,
    \]
    which matches \eqref{alg:mass}.

    Finally, we establish property (\romannumeral3). By substituting the definitions of $y_{k+1}$, $z_k$, and $z_{k+1}$ into the update rule for $x_{k+1}$ in \eqref{alg:sum}, we can rewrite \eqref{alg:sum} via a single line:
    \begin{equation}\label{eq:sum-1}
    \begin{aligned}
        x_{k+1} = & x_k - \alpha g_k + \beta(x_k - s \alpha g_k - x_{k-1} + s \alpha g_{k-1})\\
        = & x_k + \beta (x_k - x_{k-1}) - (1+s\beta)\alpha g_k+ s\beta \alpha g_{k-1}.
    \end{aligned} 
    \end{equation}
    On the other hand, rearranging \eqref{eq:sgdm-0} in a similar way as \eqref{eq:sum-1}, we obtain
    \[
     x_{k+1} = x_k + \frac{\beta_k \eta_k}{\eta_{k-1}}(x_k - x_{k-1}) - (\gamma_k+(1-\beta_k)\eta_k) g_k + \frac{\beta_k \eta_k}{\eta_{k-1}}\gamma_{k-1} g_{k-1},
    \]
    which coincides with \eqref{eq:sum-1} when $\beta_k\equiv\beta$, $\gamma_k\equiv s\alpha$, and $\eta_k\equiv\alpha/(1-\beta)-s\alpha$. The proof is now complete.
\end{proof}

Having established that \eqref{alg:g-sgdm} unifies these existing methods, we now specify the assumptions under which we will analyze its convergence properties throughout the paper.
\begin{assumption}\label{ass:s-bv}
    The objective function $f$ is $L$-smooth. Moreover, the gradient approximations $\{g_k\}$ are unbiased and have uniformly bounded variance $\sigma^2\ge 0$, i.e.,
    \[
    \ebb_k[g_k]=\nabla f(x_k),\quad \ebb_k[\|g_k-\nabla f(x_k)\|^2]\leq \sigma^2,~\forall k\ge 1.
    \]
\end{assumption}
Observe that from Assumption \ref{ass:s-bv} we have
\begin{equation}\label{eq:g-bound}
    \ebb[\|g_k\|^2] = \ebb[\| \nabla f(x_k) \|^2]+\ebb[\|g_k-\nabla f(x_k)\|^2]\leq \ebb[\| \nabla f(x_k) \|^2]+\sigma^2.
\end{equation}

\section{Convergence for Convex Problems}
\label{sec:convergence-cvx}
In this section, we establish convergence guarantees for \eqref{alg:g-sgdm} when applied to convex optimization problems. We first analyze the case of constant parameters and then derive accelerated convergence rates by employing carefully designed time-varying parameters. Throughout this section, we assume that the objective function $f$ in \eqref{eq:main-prob} admits a global minimizer $x^*\in\arg\min_{x\in\rbb^d} f(x)$.

\subsection{Convergence Results with Constant Parameters}
To establish convergence of \eqref{alg:g-sgdm} with constant parameters $\beta_k\equiv\beta\in[0,1), \gamma_k\equiv\gamma\ge0$, and $\eta_k\equiv\eta\in\rbb$, we first introduce an auxiliary sequence $\{w_k\}_{k\ge 1}$ defined by
\begin{equation}\label{def:wk}
w_k= \begin{cases}x_1, \quad \text{if } k=1 &  \\ \frac{1}{1-\beta} (x_k-\beta x_{k-1}+\beta \gamma g_{k-1}), & \text{if } k \ge 2.\end{cases}
\end{equation}

The following lemma shows that this auxiliary sequence $\{w_k\}_{k\ge 1}$ evolves according to a simple SGD update rule, which facilitates the convergence analysis significantly.
\begin{lemma}\label{lem:w-sgd}
Let the sequence $\{w_k\}_{k\ge 1}$ be defined in \eqref{def:wk}, where $\{x_k\}_{k\ge 1}$ are the iterates generated by \eqref{alg:g-sgdm} with constant parameters $\beta\in [0,1)$, $\gamma\ge 0$, and $\eta\in\rbb$. Then for any $k\ge 1$, we have
\[
w_{k+1}=w_k-(\gamma+\eta) g_k.
\]
\end{lemma}
\begin{proof}[Proof of Lemma \ref{lem:w-sgd}]
Let us first verify the case $k=1$. Note that $m_1=(1-\beta)g_1$ since $m_0=0$, hence we know that $x_2 - x_1 = - \gamma g_1-(1-\beta)\eta g_1$ from the second equation in \eqref{alg:g-sgdm}, which further implies that
\begin{align*}
w_2-w_1 & = \frac{1}{1-\beta}(x_2-\beta x_1+\beta \gamma g_1)-x_1 = \frac{1}{1-\beta}(x_2- x_1+\beta \gamma g_1) \\
& =\frac{1}{1-\beta}(-\gamma g_1-(1-\beta)\eta g_1+\beta \gamma g_1)=-(\gamma+\eta) g_1.
\end{align*}
For $k\ge 2$, it follows from \eqref{def:wk}, together with the update rules for \eqref{alg:g-sgdm} that
\begin{align*}
w_{k+1}-w_k & = \frac{1}{1-\beta}[x_{k+1}-\beta x_k+\beta \gamma g_k-(x_k-\beta x_{k-1}+\beta \gamma g_{k-1})] \\
& = \frac{1}{1-\beta}[x_{k+1}-x_k-\beta(x_k-x_{k-1})+\beta \gamma g_k-\beta \gamma g_{k-1}] \\
& = \frac{1}{1-\beta}(-\gamma g_k-\eta m_k+\beta \gamma g_{k-1}+\beta \eta m_{k-1}+\beta \gamma g_k-\beta \gamma g_{k-1}) \\
& = -\gamma g_k-\frac{\eta}{1-\beta}(m_k-\beta m_{k-1}) \\
& = - \gamma g_k-\eta g_k =- (\gamma+\eta) g_k.
\end{align*}
The proof is now complete.
\end{proof}

Lemma \ref{lem:w-sgd} reveals that the auxiliary sequence $\{w_k\}_{k\ge 1}$ is generated via SGD-like updates with constant step size $\gamma+\eta$, which serves as the ``effective" step size in the sense that it should be positive to ensure the convergence of \eqref{alg:g-sgdm}.
Now we are ready to present the main convergence result for \eqref{alg:g-sgdm} with constant parameters. 
\begin{theorem}\label{thm:cvx-const}
Let Assumption \ref{ass:s-bv} hold, and the function $f$ in \eqref{eq:main-prob} be convex. Let $\{x_k\}_{k\ge 1}$ be the sequence generated by \eqref{alg:g-sgdm} with constant parameters $\beta\in [0,1)$, $\gamma\ge 0$, and $\eta\in\rbb$. Denote $\bar{x}_t:=\frac{1}{t}\sum_{k=1}^t x_k$. If we choose
\[
\gamma\leq \frac{1}{L},\quad 0<\gamma+\eta\leq \frac{1}{L},
\]
then for any $t\ge 1$, it holds that
\[
\ebb[f(\bar{x}_t)]-f^* \leq \frac{\|x_1-x^*\|^2}{2(\gamma+\eta) t}+\frac{\beta (f(x_1)-f^*)}{(1-\beta) t}+\Big(\frac{3 \beta \gamma}{2(1-\beta)}+\frac{\gamma+\eta}{2}\Big) \sigma^2,
\]
\end{theorem}
\begin{proof}[Proof of Theorem \ref{thm:cvx-const}]
We start by expanding $\|w_{k+1}-x^*\|^2$ using Lemma \ref{lem:w-sgd}:
\begin{equation}\label{eq:cvx-const-w}
\begin{aligned}
\|w_{k+1}-x^*\|^2=&\|w_k-x^*\|^2+2\langle w_{k+1}-w_k, w_k-x^*\rangle+\|w_{k+1}-w_k\|^2\\
= & \|w_k-x^*\|^2-2(\gamma+\eta)\langle g_k, w_k-x^*\rangle +(\gamma+\eta)^2\|g_k\|^2.
\end{aligned}
\end{equation}
For $k\ge 2$, we know from \eqref{def:wk} and $\ebb_k[g_k]=\nabla f(x_k)$ that
\begin{equation}\label{eq:cvx-const-ip}
    \begin{aligned}
        \ebb[\langle g_k, w_k-x^*\rangle] =&\frac{1}{1-\beta} \ebb[\langle g_k, x_k-\beta x_{k-1}+\beta \gamma g_{k-1}-(1-\beta) x^*\rangle]\\
        =&\ebb[\langle g_k, x_k-x^*\rangle] + \frac{\beta}{1-\beta}\ebb[\langle g_k, x_k-x_{k-1}\rangle] + \frac{\beta\gamma}{1-\beta}\ebb[\langle g_k, g_{k-1}\rangle]\\
        =&\ebb[\langle \nabla f(x_k), x_k-x^*\rangle] + \frac{\beta}{1-\beta}\ebb[\langle \nabla f(x_k), x_k-x_{k-1}\rangle] + \frac{\beta\gamma}{1-\beta}\ebb[\langle g_k, g_{k-1}\rangle].
    \end{aligned}
\end{equation}
Since $f$ is convex and $L$-smooth, it holds from \eqref{eq: cvx-L} and $\nabla f(x^*)=0$ that
\begin{align}
\langle\nabla f(x_k), x_k-x^*\rangle &\ge \frac{1}{2 L}\|\nabla f(x_k)\|^2+f(x_k)-f^* \label{eq:cvx-const-1},\\
\langle\nabla f(x_k), x_k-x_{k-1}\rangle &\ge \frac{1}{2 L}\|\nabla f(x_k)-\nabla f(x_{k-1})\|^2+f(x_k)-f(x_{k-1})\label{eq:cvx-const-2}.
\end{align}
To bound $\ebb[\langle g_k, g_{k-1}\rangle]$, denote $\epsilon_k:=g_k-\nabla f(x_k)$ for all $k\ge 1$. Then it follows from $\ebb_k[\epsilon_k]=0$ and $\ebb_k[\|\epsilon_k\|^2] \leq \sigma^2$ that
\begin{equation}\label{eq:cvx-const-3}
\begin{aligned}
\ebb[\langle g_k, g_{k-1}\rangle] & =\ebb[\langle\nabla f(x_k), \nabla f(x_{k-1})\rangle]+\ebb[\langle\epsilon_k, \epsilon_{k-1}\rangle]+\ebb[\langle\epsilon_k, \nabla f(x_{k-1})]+\ebb[\langle\nabla f(x_k), \epsilon_{k-1})] \\
& \ge \ebb[\langle\nabla f(x_k), \nabla f(x_{k-1})\rangle]-\frac{1}{2}\ebb[\|\epsilon_k\|^2] -\frac{1}{2} \ebb[\|\epsilon_{k-1}\|^2]-\frac{1}{2}\ebb[\|\nabla f(x_k)\|^2]- \frac{1}{2}\ebb[\|\epsilon_{k-1}\|^2]\\
& \ge \ebb[\langle\nabla f(x_k), \nabla f(x_{k-1})\rangle]-\frac{3}{2} \sigma^2-\frac{1}{2} \ebb[\|\nabla f(x_k)\|^2],
\end{aligned}
\end{equation}
where the first inequality holds from the basic inequality $ab\ge -(a^2+b^2)/2$ and 
\[
\ebb[\langle\epsilon_k, \nabla f(x_{k-1})]=\ebb[\ebb_k[\langle\epsilon_k, \nabla f(x_{k-1})]]=\ebb[\langle\ebb_k[\epsilon_k], \nabla f(x_{k-1})]=0.
\]
Applying \eqref{eq:cvx-const-1}, \eqref{eq:cvx-const-2}, and \eqref{eq:cvx-const-3} on \eqref{eq:cvx-const-ip}, we obtain for $k\ge 2$ that
\begin{multline}\label{eq:cvx-const-ip2}
    \ebb[\langle g_k, w_k-x^*\rangle]\ge \frac{1}{2 L}\ebb[\|\nabla f(x_k)\|^2]+\ebb[f(x_k)-f^*] + \frac{\beta}{2L(1-\beta)}\ebb[\|\nabla f(x_k)-\nabla f(x_{k-1})\|^2] \\+ \frac{\beta}{1-\beta}\ebb[f(x_k)-f(x_{k-1})] + \frac{\beta\gamma}{1-\beta}\ebb[\langle\nabla f(x_k), \nabla f(x_{k-1})\rangle]-\frac{3\beta\gamma}{2(1-\beta)} \sigma^2-\frac{\beta\gamma}{2(1-\beta)} \ebb[\|\nabla f(x_k)\|^2].
\end{multline}
If $k=1$, then it follows from $w_1=x_1$ and \eqref{eq:cvx-const-1} that
\begin{equation}\label{eq:cvx-const-k=1}
    \ebb[\langle g_1, w_1-x^*\rangle] = \langle \nabla f(x_1), x_1-x^*\rangle\ge \frac{1}{2 L}\|\nabla f(x_1)\|^2+f(x_1)-f^*.
\end{equation}

For notational convenience, denote $\Delta_k:=f(x_k)-f^*\ge 0$. For $k\ge2$, taking expectations on both sides of \eqref{eq:cvx-const-w} and applying \eqref{eq:cvx-const-ip2} and \eqref{eq:g-bound} leads to
\begin{equation}\label{eq:cvx-const-w-2}
\begin{aligned}
&\ebb[\|w_{k+1}-x^*\|^2] = \ebb[\|w_k-x^*\|^2]-2(\gamma+\eta)\ebb[\langle g_k, w_k-x^*\rangle] +(\gamma+\eta)^2\ebb[\|g_k\|^2] \\
\leq & \ebb[\|w_k-x^*\|^2]-\frac{\gamma+\eta}{L} \ebb[\|\nabla f(x_k)\|^2]-2(\gamma+\eta) \ebb[\Delta_k] - \frac{\beta(\gamma+\eta)}{L(1-\beta)} \ebb[\|\nabla f(x_k)-\nabla f(x_{k-1})\|^2]\\
&-\frac{2 \beta(\gamma+\eta)}{1-\beta} \ebb[\Delta_k-\Delta_{k-1}] - \frac{2\beta\gamma(\gamma+\eta)}{1-\beta}\ebb[\langle\nabla f(x_k), \nabla f(x_{k-1})\rangle]+\frac{3\beta\gamma(\gamma+\eta)}{1-\beta} \sigma^2\\
&+\frac{\beta\gamma(\gamma+\eta)}{1-\beta} \ebb[\|\nabla f(x_k)\|^2]+(\gamma+\eta)^2 \ebb[\|\nabla f(x_k)\|^2] + (\gamma+\eta)^2\sigma^2\\
\leq & \ebb[\|w_k-x^*\|^2]-2(\gamma+\eta) \ebb[\Delta_k]-\frac{2 \beta(\gamma+\eta)}{1-\beta} \ebb[\Delta_k-\Delta_{k-1}]+\Big(\frac{3 \beta \gamma(\gamma+\eta)}{1-\beta}+(\gamma+\eta)^2\Big) \sigma^2 \\
& -\Big(\frac{\gamma+\eta}{L}-\frac{\beta \gamma(\gamma+\eta)}{1-\beta}-(\gamma+\eta)^2+\frac{\beta \gamma(\gamma+\eta)}{1-\beta}\Big) \ebb[\|\nabla f(x_k)\|^2]-\frac{\beta \gamma(\gamma+\eta)}{1-\beta} \ebb[\|\nabla f(x_{k-1})\|^2] \\
& -\Big(\frac{1}{L}-\gamma\Big) \frac{\beta(\gamma+\eta)}{1-\beta} \ebb[\|\nabla f(x_k)-\nabla f(x_{k-1})\|^2],
\end{aligned}
\end{equation}
where the last inequality holds by expanding the square term with coefficient $-\frac{\beta\gamma(\gamma+\eta)}{1-\beta}$ on the RHS of 
\[
-\frac{\beta(\gamma+\eta)}{L(1-\beta)} \ebb[\|\nabla f(x_k)-\nabla f(x_{k-1})\|^2] = \Big[-\frac{\beta\gamma(\gamma+\eta)}{1-\beta} - \Big(\frac{1}{L}-\gamma\Big) \frac{\beta(\gamma+\eta)}{1-\beta}\Big] \ebb[\|\nabla f(x_k)-\nabla f(x_{k-1})\|^2].
\]
Since $0<\gamma+\eta \leq 1/L$, we know that
\begin{equation}\label{eq:cvx-const-gamma+eta}
  \frac{\gamma+\eta}{L}-(\gamma+\eta)^2\ge 0,  
\end{equation}
which together with $0\leq\gamma\leq 1/L$ indicates that the coefficients of the last three terms are all non-positive, and hence for any $k\ge 2$ we have from \eqref{eq:cvx-const-w-2} that
\begin{equation}\label{eq:cvx-const-w-3}
    \ebb[\|w_{k+1}-x^*\|^2]\leq \ebb[\|w_k-x^*\|^2]-2(\gamma+\eta) \ebb[\Delta_k]-\frac{2 \beta(\gamma+\eta)}{1-\beta} \ebb[\Delta_k-\Delta_{k-1}]+\Big(\frac{3 \beta \gamma(\gamma+\eta)}{1-\beta}+(\gamma+\eta)^2\Big) \sigma^2.
\end{equation}
On the other hand, if $k=1$, then taking expectations on both sides of \eqref{eq:cvx-const-w} and applying \eqref{eq:cvx-const-k=1} and \eqref{eq:g-bound}, we obtain
\begin{equation}\label{eq:cvx-const-w-k=1}
\begin{aligned}
&\ebb[\|w_{2}-x^*\|^2] = \|x_1-x^*\|^2-2(\gamma+\eta)\ebb[\langle g_1, w_1-x^*\rangle] +(\gamma+\eta)^2\ebb[\|g_1\|^2] \\
\leq & \|x_1-x^*\|^2-\frac{\gamma+\eta}{L} \|\nabla f(x_1)\|^2-2(\gamma+\eta) \Delta_1 + (\gamma+\eta)^2 \|\nabla f(x_1)\|^2 + (\gamma+\eta)^2\sigma^2\\
\leq & \|x_1-x^*\|^2-2(\gamma+\eta) \Delta_1 + (\gamma+\eta)^2\sigma^2,
\end{aligned}
\end{equation}
where the last inequality holds due to \eqref{eq:cvx-const-gamma+eta}.

Summing \eqref{eq:cvx-const-w-k=1} and \eqref{eq:cvx-const-w-3} over $k=2,\dots,t$ and rearranging the terms gives
\begin{equation}\label{eq:cvx-const-opt}
\begin{aligned}
    2(\gamma+\eta) \sum_{k=1}^t \ebb[\Delta_k] \leq &\|x_1-x^*\|^2-\ebb[\|w_{t+1}-x^*\|^2]+\frac{2 \beta(\gamma+\eta)}{1-\beta} (\Delta_1-\ebb[\Delta_t])+\Big(\frac{3 \beta \gamma(\gamma+\eta)}{1-\beta}+(\gamma+\eta)^2\Big) t \sigma^2\\
    \leq & \|x_1-x^*\|^2+\frac{2 \beta(\gamma+\eta)}{1-\beta}(f(x_1)-f^*)+\Big(\frac{3 \beta \gamma(\gamma+\eta)}{1-\beta}+(\gamma+\eta)^2\Big) t \sigma^2,
\end{aligned}
\end{equation}
where the last inequality follows by dropping the non-positive terms. Finally, dividing both sides of \eqref{eq:cvx-const-opt} by $2(\gamma+\eta)t$ and using the convexity of $f$, we derive
\[
\ebb[f(\bar{x}_t)]-f^* \leq \frac{1}{t}\sum_{k=1}^t\ebb[\Delta_k]\leq\frac{\|x_1-x^*\|^2}{2(\gamma+\eta) t}+\frac{\beta (f(x_1)-f^*)}{(1-\beta) t}+\Big(\frac{3 \beta \gamma}{2(1-\beta)}+\frac{\gamma+\eta}{2}\Big) \sigma^2,
\]
which completes the proof.
\end{proof}

In Theorem \ref{thm:cvx-const}, the step size condition $\gamma+\eta\leq 1/L$ reveals an explicit trade-off between $\gamma$ and $\eta$, preventing them from being large simultaneously. Besides, since the upper bound on $\gamma+\eta$ is independent of $\beta$, Theorem \ref{thm:cvx-const} favors the choice $\beta=0$. We now show that when exact gradients are available, the convergence behavior of \eqref{alg:g-sgdm} improves substantially. In this regime, the restrictive coupling between $\gamma$ and $\eta$ can be relaxed, leading to a larger admissible step size region.
\begin{theorem}[Deterministic Case]\label{thm:cvx-const-deter}
Let the objective function $f$ in \eqref{eq:main-prob} be convex and $L$-smooth. Let $\{x_k\}_{k\ge 1}$ be the sequence generated by \eqref{alg:g-sgdm} with true  gradients $g_k=\nabla f(x_k)$ and constant parameters $\beta\in [0,1)$, $\gamma\ge 0$, and $\eta\in \rbb$. Denote $\bar{x}_t:=\frac{1}{t}\sum_{k=1}^t x_k$. If the step sizes $\gamma$ and $\eta$ satisfy
\begin{equation}\label{eq:cvx-deter-lr}
    \gamma \leq \frac{1}{L},\quad -\gamma<\eta \leq \frac{1}{L}+\frac{2 \beta-1}{1-\beta} \gamma,
\end{equation}
then for any $t\ge 1$, it holds that
\[
f(\bar{x}_t)-f^* \leq \frac{\|x_1-x^*\|^2}{2(\gamma+\eta) t}+\frac{\beta (f(x_1)-f^*+\gamma\|\nabla f(x_1)\|^2/2)}{(1-\beta) t}.
\]
\end{theorem}

\begin{proof}[Proof of Theorem \ref{thm:cvx-const-deter}]
Imitating the deductions of \eqref{eq:cvx-const-w} and \eqref{eq:cvx-const-ip} with $g_k=\nabla f(x_k)$, we have for $k\ge 2$ that
\begin{equation}\label{eq:cvx-deter}
\begin{aligned}
& \|w_{k+1}-x^*\|^2=\|w_k-x^*\|^2+2\langle w_{k+1}-w_k, w_k-x^*\rangle+\|w_{k+1}-w_k\|^2\\
= & \|w_k-x^*\|^2-2(\gamma+\eta)\langle \nabla f(x_k), w_k-x^*\rangle +(\gamma+\eta)^2\|\nabla f(x_k)\|^2\\
= & \|w_k-x^*\|^2-\frac{2(\gamma+\eta)}{1-\beta}\langle \nabla f(x_k), x_k-\beta x_{k-1}+\beta \gamma \nabla f(x_{k-1})-(1-\beta) x^*\rangle +(\gamma+\eta)^2\|\nabla f(x_k)\|^2\\
 = & \|w_k-x^*\|^2-2(\gamma+\eta)\langle \nabla f(x_k), x_k-x^*\rangle -\frac{2\beta(\gamma+\eta)}{1-\beta}\langle \nabla f(x_k), x_k-x_{k-1}\rangle\\
 & - \frac{2\beta\gamma(\gamma+\eta)}{1-\beta}\langle \nabla f(x_k), \nabla f(x_{k-1})\rangle +(\gamma+\eta)^2\|\nabla f(x_k)\|^2.
\end{aligned}
\end{equation}
Applying \eqref{eq:cvx-const-1} and \eqref{eq:cvx-const-2} on \eqref{eq:cvx-deter} and rearranging the terms further gives
\begin{align}
    &\|w_{k+1}-x^*\|^2 - \|w_k-x^*\|^2 \notag\\
    \leq& - \frac{\gamma+\eta}{L}\|\nabla f(x_k)\|^2 - 2(\gamma+\eta)(f(x_k)-f^*)-\frac{\beta(\gamma+\eta)}{L(1-\beta)}\|\nabla f(x_k)-\nabla f(x_{k-1})\|^2\notag\\
    &- \frac{2 \beta(\gamma+\eta)}{1-\beta}(f(x_k)-f(x_{k-1}))- \frac{2\beta\gamma(\gamma+\eta)}{1-\beta}\langle \nabla f(x_k), \nabla f(x_{k-1})\rangle +(\gamma+\eta)^2\|\nabla f(x_k)\|^2\notag\\
    =&-2 (\gamma+\eta)(f(x_k)-f^*)-\frac{2 \beta(\gamma+\eta)}{1-\beta}(f(x_k)-f(x_{k-1})) -\Big(\frac{\gamma+\eta}{L}-(\gamma+\eta)^2+\frac{\beta \gamma(\gamma+\eta)}{1-\beta}\Big)\|\nabla f(x_k)\|^2\notag\\
    &-\frac{\beta \gamma(\gamma+\eta)}{1-\beta}\|\nabla f(x_{k-1})\|^2-\Big(\frac{1}{L}-\gamma\Big) \frac{\beta(\gamma+\eta)}{1-\beta} \|\nabla f(x_k)-\nabla f(x_{k-1})\|^2\label{eq:cvx-deter-1}.
\end{align}
From the step size condition \eqref{eq:cvx-deter-lr}, we know that $\gamma+\eta>0$ and hence
\begin{equation}\label{eq:cvx-deter-gamma+eta}
    \begin{aligned}
    &\frac{\gamma+\eta}{L}+\frac{\beta \gamma(\gamma+\eta)}{1-\beta}-(\gamma+\eta)^2 = (\gamma+\eta)\Big(\frac{1}{L}+\frac{\beta \gamma}{1-\beta}-\gamma-\eta\Big)\\
    \ge& (\gamma+\eta)\Big(\frac{1}{L}+\frac{\beta \gamma}{1-\beta}-\gamma-\frac{1}{L}-\frac{2 \beta-1}{1-\beta} \gamma\Big)=0.
\end{aligned}
\end{equation}
It then follows that the coefficients of the last three terms on the RHS of \eqref{eq:cvx-deter-1} are all non-positive due to $0\leq\gamma\leq 1/L$. Then for $k\ge 2$, we further obtain from \eqref{eq:cvx-deter-1} that
\begin{equation}\label{eq:cvx-deter-2}
    \|w_{k+1}-x^*\|^2 - \|w_k-x^*\|^2 \leq -2 (\gamma+\eta)(f(x_k)-f^*)-\frac{2 \beta(\gamma+\eta)}{1-\beta}(f(x_k)-f(x_{k-1})).
\end{equation}
On the other hand, if $k=1$, then we know from $w_1=x_1$ and \eqref{eq:cvx-const-1} that
\begin{equation}\label{eq:cvx-deter-k=1}
    \begin{aligned}
        \|w_2-x^*\|^2 = &\|x_1-x^*\|^2 - 2(\gamma+\eta)\langle \nabla f(x_1), x_1-x^*\rangle +(\gamma+\eta)^2\|\nabla f(x_1)\|^2\\
        \leq & \|x_1-x^*\|^2 - \frac{\gamma+\eta}{L}\|\nabla f(x_1)\|^2 - 2(\gamma+\eta)(f(x_1)-f^*) +(\gamma+\eta)^2\|\nabla f(x_1)\|^2\\
        \leq & \|x_1-x^*\|^2 - 2(\gamma+\eta)(f(x_1)-f^*) +\frac{\beta \gamma(\gamma+\eta)}{1-\beta}\|\nabla f(x_1)\|^2,
    \end{aligned}
\end{equation}
where the last inequality holds from \eqref{eq:cvx-deter-gamma+eta}.

Summing \eqref{eq:cvx-deter-k=1} and \eqref{eq:cvx-deter-2} over $k=2, \dots, t$ and rearranging the terms gives
\begin{equation}\label{eq:cvx-deter-3}
\begin{aligned}
    &2(\gamma+\eta) \sum_{k=1}^t (f(x_k)-f^*)\\
    \leq & \|x_1-x^*\|^2-\|w_{t+1}-x^*\|^2+\frac{2 \beta(\gamma+\eta)}{1-\beta} (f(x_1)-f(x_t))+\frac{\beta \gamma(\gamma+\eta)}{1-\beta}\|\nabla f(x_1)\|^2\\
    \leq & \|x_1-x^*\|^2+\frac{\beta(\gamma+\eta)}{1-\beta}\big[2(f(x_1)-f^*)+\gamma\|\nabla f(x_1)\|^2\big],
\end{aligned}
\end{equation}
where the last inequality follows from $f(x_t)\ge f^*$.
Finally, dividing both sides of \eqref{eq:cvx-deter-3} by $2(\gamma+\eta)t$ and using the convexity of $f$ leads to
\[
f(\bar{x}_t)-f^* \leq \frac{1}{t}\sum_{k=1}^t (f(x_k)-f^*)\leq\frac{\|x_1-x^*\|^2}{2(\gamma+\eta) t}+\frac{\beta}{(1-\beta)t}\big(f(x_1)-f^* +\gamma\|\nabla f(x_1)\|^2/2\big),
\]
which completes the proof.
\end{proof}
\begin{remark}
    Note that while $\gamma = \frac{1}{L}$ is a standard choice for vanilla gradient descent, the step size $\eta$ in our framework can exceed $1/L$ when $\beta>0.5$. This partially explains the benefit of incorporating momentum. In particular, the maximal admissible value $\frac{1}{L}+\frac{2\beta-1}{1-\beta}\gamma$ is monotonically increasing in $\beta$. Consequently, substantially larger step sizes $\eta$ can be employed when $\beta$ is chosen close to one, as is common in practice (e.g., $\beta=0.9$ or 0.99), potentially leading to faster convergence.
\end{remark}

\begin{remark}
By setting $\gamma=0$ in Theorem \ref{thm:cvx-const-deter}, we obtain an $O(1/t)$ ergodic convergence rate for the heavy-ball (HB) method (i.e., \eqref{alg:sgd-hb} with constant parameters) for any $\beta\in[0,1)$ and $0<\eta\leq 1/L$:
\begin{equation}\label{eq:conv-hb-deter}
    f(\bar{x}_t)-f^* \leq \frac{\|x_1-x^*\|^2}{2\eta t}+\frac{\beta (f(x_1)-f^*)}{(1-\beta) t}.
\end{equation}
Moreover, since $\gamma\leq 1/L$, choosing $\eta=\frac{\beta\gamma}{1-\beta}$ satisfies the step size condition \eqref{eq:cvx-deter-lr} as $\frac{\beta\gamma}{1-\beta}=\gamma+\frac{2\beta-1}{1-\beta}\gamma\leq \frac{1}{L}+\frac{2\beta-1}{1-\beta}\gamma$.
Therefore, substituting $\eta=\frac{\beta\gamma}{1-\beta}$ into Theorem \ref{thm:cvx-const-deter} yields an $O(1/t)$ ergodic convergence rate for Nesterov's accelerated gradient (NAG) method (i.e., \eqref{alg:sgd-nesterov} with constant parameters) for any $\beta\in[0,1)$ and $0<\gamma\leq 1/L$:
\begin{equation}\label{eq:conv-nes-deter}
  f(\bar{x}_t)-f^* \leq \frac{(1-\beta)\|x_1-x^*\|^2}{2\gamma t}+\frac{\beta (f(x_1)-f^*+\gamma\|\nabla f(x_1)\|^2/2)}{(1-\beta) t}. 
\end{equation}
Similar convergence results for HB in \eqref{eq:conv-hb-deter} and NAG in \eqref{eq:conv-nes-deter} were previously established by \citet{ghadimi2015global} via separate analyses. In contrast, these bounds arise here as direct consequences of the unified analysis developed in Theorem \ref{thm:cvx-const-deter}. Furthermore, convergence guarantees for a broad class of momentum schemes can be derived for both deterministic settings (as in Theorem \ref{thm:cvx-const-deter}) and stochastic settings (as in Theorem \ref{thm:cvx-const}).
\end{remark}

\subsection{Improved Convergence Results with Time-varying Parameters}
While Theorems \ref{thm:cvx-const} and \ref{thm:cvx-const-deter} establish convergence guarantees for a broad class of constant parameter choices, the resulting rates are limited to $O(1/t)$. Consequently, the well-known optimal rate $O(1/t^2)$ achieved by NAG is not covered by these results. This limitation stems from the fact that the $O(1/t^2)$ rate is attained through the use of time-varying momentum parameters rather than constant ones.
In this subsection, we therefore turn to the case of time-varying parameters and show that, under suitable conditions, \eqref{alg:g-sgdm} enjoys improved convergence guarantees. In particular, our analysis recovers the accelerated $O(1/t^2)$ rate of NAG as a direct consequence.

Recall that we construct an auxiliary sequence $\{w_k\}_{k\ge 1}$ defined in \eqref{def:wk} to facilitate the convergence analysis when using constant parameters in \eqref{alg:g-sgdm}. However, a similar construction of $\{w_k\}_{k\ge 1}$ with time-varying parameters $\beta_k, \gamma_k$, and $\eta_k$ replacing their constant counterparts is generally not useful, because the SGD-like property (see Lemma \ref{lem:w-sgd}) no longer holds. To preserve this property, we impose the following condition on the parameters: there exists a sequence $\{\theta_k\}_{k\ge 1}$ with $\theta_k\in (0, 1)$ such that
\begin{equation}\label{eq:cvx-cond-beta}
\frac{\beta_k\eta_k}{\eta_{k-1}}=\theta_{k+1}\Big(\frac{1}{\theta_k}-1\Big),~\forall k\ge 2.
\end{equation}
It is straightforward to verify that $\beta_k\in [0, 1)$ holds under condition \eqref{eq:cvx-cond-beta} in many cases (e.g., with constant $\eta_k\equiv\eta\neq0$ and non-increasing $\{\theta_k\}_{k\ge 1}$). Using $\{\theta_k\}_{k\ge 1}$, we define an auxiliary sequence $\{v_k\}_{k\ge 1}$ as follows:
\begin{equation}\label{def:vk}
v_k= \begin{cases} x_1, \quad \text{if } k=1 &  \\ \frac{1}{\theta_k}\big(x_k-(1-\theta_k) x_{k-1}+(1-\theta_k) \gamma_{k-1} g_{k-1}\big), & \text{if } k \ge 2.\end{cases}
\end{equation}
We also introduce an intermediate sequence $\{y_k\}_{k\ge 1}$ defined by
\begin{equation}\label{def:yk}
y_k= \begin{cases} x_1, \quad \text{if } k=1 &  \\ x_{k-1}-\gamma_{k-1}g_{k-1}, & \text{if } k \ge 2.\end{cases}
\end{equation}
In the upcoming lemma, we show that $\{v_k\}_{k\ge 1}$ exhibits an SGD-like behavior, similar to Lemma~\ref{lem:w-sgd}. Furthermore, the difference $v_k-x_k$ is proportional to $x_k-y_k$. These properties provide crucial technical support for establishing refined convergence results.
\begin{lemma}\label{lem:v-sgd}
Consider the sequences $\{v_k\}_{k\ge 1}$ and $\{y_k\}_{k\ge 1}$ defined in \eqref{def:vk} and \eqref{def:yk} respectively, where $\{x_k\}_{k\ge 1}$ are the iterates generated by \eqref{alg:g-sgdm}. Let the condition \eqref{eq:cvx-cond-beta} hold.
Then for any $k\ge 1$, it holds that
\begin{align}
    &v_{k+1}=v_k-\Big(\gamma_k+\frac{1-\beta_k}{\theta_{k+1}}\eta_k\Big) g_k,\label{eq:v-sgd}\\
    &v_k-x_k = \Big(\frac{1}{\theta_k}-1\Big)(x_k-y_k)\label{eq:v-x}
\end{align}
\end{lemma}
\begin{proof}[Proof of Lemma \ref{lem:v-sgd}]
For $k=1$, we have $m_1=(1-\beta_1)g_1$ since $m_0=0$, and hence $x_2 - x_1 = - \gamma_1 g_1-(1-\beta_1)\eta_1 g_1$ from the second equation in \eqref{alg:g-sgdm}, which further implies that
\begin{align*}
v_2-v_1 & = \frac{1}{\theta_2}(x_2-(1-\theta_2) x_1+(1-\theta_2) \gamma_1 g_1)- x_1 = \frac{1}{\theta_2}(x_2- x_1+(1-\theta_2) \gamma_1 g_1) \\
& =\frac{1}{\theta_2}(-\gamma_1 g_1-(1-\beta_1)\eta_1 g_1+(1-\theta_2) \gamma_1 g_1)=-\Big(\gamma_1+\frac{1-\beta_1}{\theta_2}\eta_1\Big)g_1.
\end{align*}
Hence \eqref{eq:v-sgd} holds with $k=1$.
For $k\ge 2$, it follows from \eqref{def:vk} and the update rules of \eqref{alg:g-sgdm} that
\begin{align*}
v_{k+1}-v_k =& \frac{1}{\theta_{k+1}}\big(x_{k+1}-(1-\theta_{k+1}) x_k+(1-\theta_{k+1}) \gamma_k g_k\big)-\frac{1}{\theta_k}\big(x_k-(1-\theta_k) x_{k-1}+(1-\theta_k) \gamma_{k-1} g_{k-1}\big)\\
= & \frac{1}{\theta_{k+1}}\big(x_{k+1}- x_k+(1-\theta_{k+1}) \gamma_k g_k\big)-\frac{1}{\theta_k}\big(x_k-x_{k-1}+(1-\theta_k) \gamma_{k-1} g_{k-1}\big)+x_k - x_{k-1}\\
= & \frac{1}{\theta_{k+1}}\big(-\theta_{k+1} \gamma_k g_k - \eta_k m_k\big)-\frac{1}{\theta_k}\big(-\theta_k \gamma_{k-1} g_{k-1} - \eta_{k-1} m_{k-1}\big) - \gamma_{k-1}g_{k-1} - \eta_{k-1} m_{k-1} \\
 = & - \gamma_k g_k - \frac{\eta_k}{\theta_{k+1}} \Big(m_k - \frac{\theta_{k+1}}{\eta_k}\Big(\frac{1}{\theta_k}-1\Big)\eta_{k-1} m_{k-1}\Big)\\
 = & - \gamma_k g_k - \frac{\eta_k}{\theta_{k+1}} (m_k - \beta_k m_{k-1})\\
 = & - \gamma_k g_k - \frac{\eta_k}{\theta_{k+1}}(1-\beta_k)g_k,
\end{align*}
where we apply \eqref{eq:cvx-cond-beta} to obtain the second last equality. Therefore, \eqref{eq:v-sgd} holds for all $k\ge 1$. 

We proceed to prove \eqref{eq:v-x}. Note it holds trivially for $k=1$ since $v_1 = y_1 = x_1$, hence we consider $k\ge 2$ in what follows. From the definitions of $v_k$ and $y_k$, together with the update rules of \eqref{alg:g-sgdm}, we derive
\begin{align*}
    v_k-x_k = & \frac{1}{\theta_k}\big(x_k-(1-\theta_k) x_{k-1}+(1-\theta_k) \gamma_{k-1} g_{k-1}\big) - x_k\\
    = & \Big(\frac{1}{\theta_k}-1\Big)x_k-\frac{1-\theta_k}{\theta_k}\big(x_{k-1} - \gamma_{k-1} g_{k-1}\big)= \Big(\frac{1}{\theta_k}-1\Big)(x_k - y_k).
\end{align*}
The proof is now complete.
\end{proof}

Lemma \ref{lem:v-sgd} indicates that the ``effective" step sizes are $\{\gamma_k+\frac{1-\beta_k}{\theta_{k+1}}\eta_k\}_{k\ge 1}$, which should be positive to ensure the convergence of \eqref{alg:g-sgdm}.
In the upcoming theorem, we present the refined convergence result of \eqref{alg:g-sgdm} at the intermediate point $y_t$. To establish this result, we introduce a sequence of Lyapunov functions $\{\Phi_k\}_{k\ge 1}$ defined by
\begin{equation}\label{def:Phi_k}
    \Phi_k=\frac{1}{\Theta_{k-1}}(f(y_k)-f^*)+p_k\|v_k-x^*\|^2,
\end{equation}
where $\{p_k\}_{k\ge 1}$ is a positive and non-increasing sequence to be specified in the proof, and $\{\Theta_k\}_{k\ge 0}$ is defined recursively via
\begin{equation}\label{def:Theta}
    \Theta_0=1,~\text{and}~\Theta_k := (1-\theta_k)\Theta_{k-1},~\forall k\ge 1.
\end{equation}
We shall show that if the sequence $\{p_k\}_{k\ge 1}$ is chosen appropriately, then the differences between consecutive Lyapunov functions can be bounded from above in expectation by terms proportional to $\sigma^2$, which is the key to our proof. 
\begin{theorem}\label{thm:cvx-vary}
Let Assumption \ref{ass:s-bv} hold, and the objective function $f$ in \eqref{eq:main-prob} be convex. Let $\{x_k\}_{k\ge 1}$ be the sequence generated by \eqref{alg:g-sgdm}. Suppose the parameters $\{\beta_k, \gamma_k, \eta_k\}_{k\ge 1}$ are chosen such that \eqref{eq:cvx-cond-beta} and the following conditions hold for all $k\ge 1$:
\begin{align}
    &\gamma_k+\frac{1-\beta_k}{\theta_{k+1}}\eta_k> 0,\label{eq:cvx-lr-0}\\
    &2-L \gamma_k-\frac{\theta_k(1-\beta_k)\eta_k}{\theta_{k+1}\gamma_k}-\theta_k\ge 0,\label{eq:cvx-lr-1}\\
    &\frac{\theta_k}{\theta_{k+1}\gamma_k+(1-\beta_k)\eta_k} \ge \frac{\theta_{k+2}}{(1-\theta_{k+1})(\theta_{k+2}\gamma_{k+1}+(1-\beta_{k+1})\eta_{k+1})}.\label{eq:cvx-lr-2}
\end{align}
Then, for any $t\ge 1$, it holds that
\begin{equation}\label{eq:rate-general}
    \ebb[f(y_t)]-f^*\leq \Phi_1\Theta_{t-1}+\sigma^2\sum_{k=1}^{t-1}\frac{\Theta_{t-1}}{2\Theta_k}\Big(L \gamma_k^2+\theta_k\Big(\gamma_k+\frac{1-\beta_k}{\theta_{k+1}}\eta_k\Big)\Big),
\end{equation}
where $y_t$ is defined in \eqref{def:yk} and 
\begin{equation}\label{def:Phi_1}
    \Phi_1:=f(x_1)-f^*+\frac{\theta_1\theta_2}{2(1-\theta_1)(\theta_2\gamma_1+(1-\beta_1)\eta_1)}\|x_1-x^*\|^2.
\end{equation}
\end{theorem}

\begin{proof}[Proof of Theorem \ref{thm:cvx-vary}]
For notational convenience, denote the effective step sizes as
\begin{equation}\label{def:tilde_gamma}
    \tilde{\gamma}_k:=\gamma_k+\frac{1-\beta_k}{\theta_{k+1}}\eta_k>0,
\end{equation}
where the positivity is guaranteed by \eqref{eq:cvx-lr-0}. Then it follows from \eqref{eq:v-sgd} in Lemma \ref{lem:v-sgd} that
\begin{equation}\label{eq:diff-v}
    v_{k+1}=v_k- \tilde{\gamma}_k g_k.
\end{equation}
Consider the Lyapunov functions $\{\Phi_k\}_{k\ge 1}$ given by \eqref{def:Phi_k} with 
\begin{equation}\label{def:pk}
   p_k:=\frac{\theta_k}{2\tilde{\gamma}_k\Theta_k}>0.
\end{equation}
Then we know from \eqref{def:tilde_gamma} and the condition \eqref{eq:cvx-lr-2} that
\begin{align*}
    p_{k+1}=&\frac{\theta_{k+1}}{2\tilde{\gamma}_{k+1}\Theta_{k+1}}=\frac{\theta_{k+1}}{2(1-\theta_{k+1})(\gamma_{k+1}+\frac{1-\beta_{k+1}}{\theta_{k+2}}\eta_{k+1})\Theta_k}\\
    =&\frac{\theta_{k+1}}{2\Theta_k}\frac{\theta_{k+2}}{(1-\theta_{k+1})(\theta_{k+2}\gamma_{k+1}+(1-\beta_{k+1})\eta_{k+1})}\\
    \leq & \frac{\theta_{k+1}}{2\Theta_k}\frac{\theta_k}{\theta_{k+1}\gamma_k+(1-\beta_k)\eta_k} = \frac{\theta_k}{2\tilde{\gamma}_k\Theta_k}=p_k,
\end{align*}
which implies that $\{p_k\}_{k\ge 1}$ is non-increasing. Denote $\Delta \Phi_k:=\Phi_{k+1}-\Phi_k$, then it follows that
\begin{equation}\label{eq:diff-phi-1}
\Delta \Phi_k \leq \frac{1}{\Theta_{k-1}}(f(y_{k+1})-f(y_k))+\Big(\frac{1}{\Theta_k}-\frac{1}{\Theta_{k-1}}\Big)(f(y_{k+1})-f^*)+p_k\big(\|v_{k+1}-x^*\|^2-\|v_k-x^*\|^2\big).
\end{equation}
To bound the last term on the RHS of \eqref{eq:diff-phi-1}, we derive from \eqref{eq:diff-v} that 
\begin{equation}\label{eq:diff-v-1}
\|v_{k+1}-x^*\|^2-\|v_k-x^*\|^2=\|v_{k+1}-v_k\|^2+2\langle v_{k+1}-v_k, v_k-x^*\rangle = \tilde{\gamma}_k^2\|g_k\|^2+2 \tilde{\gamma}_k\langle g_k, x^*-v_k\rangle.
\end{equation}
From the $L$-smoothness of $f$ and \eqref{def:yk}, we know that
\begin{equation}\label{eq:f_yt}
\begin{aligned}
f(y_{k+1}) & \leq f(x_k)+\langle\nabla f(x_k), y_{k+1}-x_k\rangle+\frac{L}{2}\|y_{k+1}-x_k\|^2 \\
& =f(x_k)-\gamma_k\langle\nabla f(x_k), g_k\rangle+\frac{L \gamma_k^2}{2}\|g_k\|^2.
\end{aligned}
\end{equation}
Plugging \eqref{eq:diff-v-1} and \eqref{eq:f_yt} into \eqref{eq:diff-phi-1}, together with the convexity of $f$ gives
\begin{equation}\label{eq:diff-phi-2}
\begin{aligned}
\Delta \Phi_k \leq & \frac{1}{\Theta_{k-1}}(f(x_k)-f(y_k))+\Big(\frac{1}{\Theta_k}-\frac{1}{\Theta_{k-1}}\Big)(f(x_k)-f^*)-\frac{\gamma_k}{\Theta_k}\langle\nabla f(x_k), g_k\rangle\\
&+\Big(\frac{L \gamma_k^2}{2 \Theta_k}+p_k \tilde{\gamma}_k^2\Big)\|g_k\|^2 +2 p_k \tilde{\gamma}_k\langle g_k, x^*-v_k\rangle \\
\leq & \frac{1}{\Theta_{k-1}}\langle\nabla f(x_k), x_k-y_k\rangle+\Big(\frac{1}{\Theta_k}-\frac{1}{\Theta_{k-1}}\Big)\langle\nabla f(x_k), x_k-x^*\rangle -\frac{\gamma_k}{\Theta_k}\langle\nabla f(x_k), g_k\rangle\\
& +\Big(\frac{L \gamma_k^2}{2 \Theta_k}+p_k \tilde{\gamma}_k^2\Big)\|g_k\|^2+2 p_k \tilde{\gamma}_k\langle g_k, x^*-v_k\rangle.
\end{aligned}
\end{equation}
Note from Assumption \ref{ass:s-bv} that
\begin{align*}
    &\ebb[\langle g_k, x^*-v_k\rangle]=\ebb[\langle \nabla f(x_k), x^*-v_k\rangle],\\
    &\ebb[\langle\nabla f(x_k), g_k\rangle] = \ebb[\|\nabla f(x_k)\|^2],\\
    &\ebb[\|g_k\|^2]\leq \ebb[\|\nabla f(x_k)\|^2] + \sigma^2.
\end{align*}
Therefore, taking expectations on both sides of \eqref{eq:diff-phi-2} further leads to
\begin{multline}\label{eq:diff-phi-3}
\ebb[\Delta\Phi_k] \leq \ebb\Big[\Big\langle\nabla f(x_k), \frac{1}{\Theta_k} x_k-\frac{1}{\Theta_{k-1}} y_k-2 p_k \tilde{\gamma}_k v_k+\Big(2 p_k \tilde{\gamma}_k-\frac{1}{\Theta_k}+\frac{1}{\Theta_{k-1}}\Big) x^*\Big\rangle\Big] \\
-\Big(\frac{\gamma_k}{\Theta_k}-\frac{L \gamma_k^2}{2 \Theta_k}-p_k \tilde{\gamma}_k^2\Big) \ebb[\|\nabla f(x_k)\|^2]+\Big(\frac{L \gamma_k^2}{2 \Theta_k}+p_k \tilde{\gamma}_k^2\Big) \sigma^2.
\end{multline}

To proceed, note from the definitions of $p_k$ \eqref{def:pk} and $\Theta_k$ \eqref{def:Theta} that
\begin{equation}\label{eq:pt=0}
    2 p_k \tilde{\gamma}_k -\frac{1}{\Theta_k}+\frac{1}{\Theta_{k-1}} = \frac{\theta_k}{\Theta_k}-\frac{1}{\Theta_k}+\frac{1}{\Theta_{k-1}} = \frac{\theta_k-1}{(1-\theta_k)\Theta_{k-1}}+\frac{1}{\Theta_{k-1}}=0.
\end{equation}
Moreover, we obtain from \eqref{eq:v-x} and \eqref{def:Theta} that
\begin{equation}\label{eq:x-y-v}
\begin{aligned}
   \frac{1}{\Theta_k} x_k - \frac{1}{\Theta_{k-1}} y_k-2 p_k \tilde{\gamma}_k v_k = & \frac{1}{\Theta_k} x_k - \frac{1}{\Theta_{k-1}} y_k- \frac{\theta_k}{\Theta_k}v_k\\
   =&\frac{1}{\Theta_k}(x_k-(1-\theta_k)y_k-\theta_k v_k)\\
   =&\frac{1}{\Theta_k}\Big(x_k-(1-\theta_k)y_k-\theta_k (x_k+(1/\theta_k-1)(x_k-y_k))\Big) = 0,
\end{aligned}
\end{equation}
which together with \eqref{eq:pt=0} indicates that the inner product term in \eqref{eq:diff-phi-3} vanishes. Moreover, it follows from \eqref{def:tilde_gamma}, condition \eqref{eq:cvx-lr-1}, and the definition of $p_k$ \eqref{def:pk} that
\begin{equation}\label{eq:term-2}
\begin{aligned}
    \frac{\gamma_k}{\Theta_k}-\frac{L \gamma_k^2}{2 \Theta_k}-p_k \tilde{\gamma}_k^2 &= \frac{\gamma_k}{\Theta_k}-\frac{L \gamma_k^2}{2\Theta_k}-\frac{\theta_k}{2\Theta_k}\Big(\gamma_k+\frac{1-\beta_k}{\theta_{k+1}}\eta_k\Big) \\
    &= \frac{\gamma_k}{2\Theta_k}\Big(2-L \gamma_k-\frac{\theta_k(1-\beta_k)\eta_k}{\theta_{k+1}\gamma_k}-\theta_k\Big)\ge 0.
\end{aligned}
\end{equation}
Combining \eqref{eq:pt=0}, \eqref{eq:x-y-v}, and \eqref{eq:term-2}, we obtain from \eqref{eq:diff-phi-3} that
\[
    \ebb[\Delta\Phi_k] \leq \Big(\frac{L \gamma_k^2}{2 \Theta_k}+p_k \tilde{\gamma}_k^2\Big) \sigma^2 = \Big(\frac{L \gamma_k^2}{2 \Theta_k}+\frac{\theta_k}{2\Theta_k}\Big(\gamma_k+\frac{1-\beta_k}{\theta_{k+1}}\eta_k\Big)\Big)\sigma^2 = \frac{1}{2\Theta_k}\Big(L \gamma_k^2+\theta_k\Big(\gamma_k+\frac{1-\beta_k}{\theta_{k+1}}\eta_k\Big)\Big)\sigma^2,
\]
which further gives
\begin{align*}
    \frac{1}{\Theta_{t-1}}\ebb[f(y_t)-f^*]&\leq\ebb[\Phi_t] = \Phi_1+\sum_{k=1}^{t-1}\ebb[\Delta\Phi_k]\\
    &\leq \Phi_1+\sigma^2\sum_{k=1}^{t-1}\frac{1}{2\Theta_k}\Big(L \gamma_k^2+\theta_k\Big(\gamma_k+\frac{1-\beta_k}{\theta_{k+1}}\eta_k\Big)\Big).
\end{align*}
Moreover, we know from $\Theta_0=1$ and \eqref{def:pk} that
\[
p_1:=\frac{\theta_1}{2\tilde{\gamma}_1\Theta_1}=\frac{\theta_1}{2(1-\theta_1)(\gamma_1+\frac{1-\beta_1}{\theta_2}\eta_1)}=\frac{\theta_1\theta_2}{2(1-\theta_1)(\theta_2\gamma_1+(1-\beta_1)\eta_1)},
\]
which together with $v_1=y_1=x_1$ implies \eqref{def:Phi_1}. The proof is now complete.
\end{proof}

While Theorem \ref{thm:cvx-vary} provides general convergence bounds for \eqref{alg:g-sgdm} with time-varying parameters satisfying conditions \eqref{eq:cvx-cond-beta} and \eqref{eq:cvx-lr-0}-\eqref{eq:cvx-lr-2}, it does not yet reveal explicit convergence rates. To obtain concrete results, we next specify the parameters $\{\beta_k, \gamma_k, \eta_k\}_{k\ge 1}$, for which all the required conditions hold automatically with a particular choice of $\{\theta_k\}_{k\ge 1}$. This choice allows the bound in Theorem \ref{thm:cvx-vary} to be evaluated in closed form, leading to the optimal $O(1/t^2)$ convergence rate in the deterministic setting and an $O(1/t^2+\sigma/\sqrt{t})$ rate in the stochastic setting, as stated in the following corollary. 
\begin{corollary}\label{cor:cvx-vary}
    Consider the same setting as in Theorem \ref{thm:cvx-vary}. Let $\gamma_k\equiv\gamma\in (0,1/L]$. Given an arbitrary $\beta_1\in [0,1)$, define $\eta_1:=\frac{L\gamma^2-\gamma}{2(1-\beta_1)}$. The parameters $\{\eta_k\}_{k\ge 1}$ and $\{\beta_k\}_{k\ge 1}$ are then generated recursively via
    \begin{align}
    &\eta_k=\frac{k}{k+3} \eta_{k-1}+\frac{(k+1) L \gamma^2-2 \gamma}{k+3},~\forall k\ge 2,\label{eq:eta-gen}\\
    &\beta_k=\frac{k\eta_{k-1}}{(k+3)\eta_k},~\forall k\ge 2.\label{eq:beta-gen}
    \end{align}
    Then, for any $t\ge 1$, the following results hold.
    \begin{itemize}
        \item[(i).] In the deterministic setting with $\sigma^2=0$, choosing $\gamma=1/L$ yields
        \[
            \ebb[f(y_t)]-f^*\leq \frac{2(f(x_1)-f^*+L\|x_1-x^*\|^2)}{t(t+1)}.
        \]
        \item[(ii).] In the stochastic setting with $\sigma^2>0$, let $C>0$ and choose
        \begin{equation}\label{eq:gamma-gen}
            \gamma=\min\Big\{\frac{1}{L},~t^{-3/4}\sqrt{\frac{C}{\sigma L}}\Big\}.
        \end{equation}
        Then we have
        \[
        \ebb[f(y_t)]-f^*\leq \frac{2(f(x_1)-f^*+L\|x_1-x^*\|^2)}{t(t+1)} + \frac{\sigma}{\sqrt{t}}\Big(\frac{\|x_1-x^*\|^2}{C}+\frac{11C}{3}\Big).
        \]
    \end{itemize}
\end{corollary}

\begin{proof}[Proof of Corollary \ref{cor:cvx-vary}]
    Let $\theta_k:=\frac{2}{k+2}\in(0,1)$. Then it is straightforward to verify the condition \eqref{eq:cvx-cond-beta} by \eqref{eq:beta-gen} and observing that
    \[
    \theta_{k+1}\Big(\frac{1}{\theta_k}-1\Big)= \theta_{k+1}/\theta_k-\theta_{k+1} = \frac{k+2}{k+3} - \frac{2}{k+3} = \frac{k}{k+3}.
    \]
    Next, we show that the conditions \eqref{eq:cvx-lr-0}, \eqref{eq:cvx-lr-1}, and \eqref{eq:cvx-lr-2} are also satisfied with the given choices of $\{\theta_k\}_{k\ge 1}$ and parameters $\{\beta_k, \gamma_k, \eta_k\}_{k\ge 1}$. To this end, first note from \eqref{eq:eta-gen} and \eqref{eq:beta-gen} that for any $k\ge 2$:
    \begin{equation}\label{eq:t-square-1}
        (1-\beta_k)\eta_k = \eta_k - \beta_k \eta_k=\eta_k-\frac{k}{k+3} \eta_{k-1}= \frac{(k+1) L \gamma^2-2 \gamma}{k+3}.
    \end{equation}
    Since $\eta_1:=\frac{L\gamma^2-\gamma}{2(1-\beta_1)}$, it follows that \eqref{eq:t-square-1} also holds for $k=1$ and hence we have
    \begin{equation}\label{eq:tilde_gamma>0}
        \gamma+\frac{1-\beta_k}{\theta_{k+1}}\eta_k = \gamma + \frac{k+3}{2}\cdot\frac{(k+1) L \gamma^2-2 \gamma}{k+3} = \frac{(k+1) L \gamma^2}{2}>0,~\forall k\ge 1,
    \end{equation}
    which verifies the condition \eqref{eq:cvx-lr-0}. Moreover, we obtain from $\gamma\leq 1/L$ and \eqref{eq:t-square-1} that
    \begin{align*}
        2-L \gamma-\frac{\theta_k(1-\beta_k)\eta_k}{\theta_{k+1}\gamma}-\theta_k=& 2-L \gamma - \frac{k+3}{k+2}\cdot\frac{(k+1) L \gamma-2}{k+3} - \frac{2}{k+2}\\
        =& 2-L\gamma -\frac{(k+1) L \gamma}{k+2} \ge 2-2L\gamma\ge 0,~\forall k\ge 1.
    \end{align*}
    Hence, condition \eqref{eq:cvx-lr-1} is satisfied. Using \eqref{eq:t-square-1}, direct computation gives 
    \begin{align*}
        &\frac{\theta_k}{\theta_{k+1}\gamma+(1-\beta_k)\eta_k} = \frac{2/(k+2)}{2\gamma/(k+3)+[(k+1)L\gamma^2-2\gamma]/(k+3)}= \frac{2(k+3)}{(k+1)(k+2)L\gamma^2},\\
        & \frac{\theta_{k+2}}{(1-\theta_{k+1})(\theta_{k+2}\gamma+(1-\beta_{k+1})\eta_{k+1})} = \frac{2/(k+4)}{\frac{k+1}{k+3}\Big(\frac{2\gamma}{k+4}+\frac{(k+2)L\gamma^2-2\gamma}{k+4}\Big)}=\frac{2(k+3)}{(k+1)(k+2)L\gamma^2},
    \end{align*}
    which indicates that \eqref{eq:cvx-lr-2} holds with equalities. 

    We proceed to prove the desired convergence rates for both cases. With the given $\{\theta_k\}_{k\ge 1}$, it follows from \eqref{def:Theta} that
    \begin{equation}\label{eq:Theta}
        \Theta_k = \frac{2}{(k+1)(k+2)},~\forall k\ge 0.
    \end{equation}
    Additionally, plugging $\theta_1=2/3$, $\theta_2=1/2$, and $(1-\beta_1)\eta_1:=(L\gamma^2-\gamma)/2$ into \eqref{def:Phi_1} gives
    \begin{equation}\label{eq:Phi_1}
    \begin{aligned}
       \Phi_1=&f(x_1)-f^*+\frac{\theta_1\theta_2}{2(1-\theta_1)(\theta_2\gamma+(1-\beta_1)\eta_1)}\|x_1-x^*\|^2\\
       =&f(x_1)-f^*+\frac{1}{L\gamma^2}\|x_1-x^*\|^2.
    \end{aligned}
    \end{equation}
    Applying \eqref{eq:Theta} with $k=t-1$ and $\sigma^2=0$ on \eqref{eq:rate-general}, together with \eqref{eq:Phi_1} and $\gamma=1/L$ leads to
    \[
        \ebb[f(y_t)]-f^*\leq \Phi_1\Theta_{t-1} = \frac{2(f(x_1)-f^*+L\|x_1-x^*\|^2)}{t(t+1)},
    \]
    which proves the deterministic case (i).

    Next, we consider the stochastic case (ii) with $\sigma^2>0$. Note from \eqref{eq:tilde_gamma>0} that
    \begin{equation*}
        L \gamma^2+\theta_k\Big(\gamma+\frac{1-\beta_k}{\theta_{k+1}}\eta_k\Big)= L\gamma^2+\frac{2}{k+2}\cdot\frac{(k+1) L \gamma^2}{2} = L\gamma^2+\frac{k+1}{k+2}L\gamma^2\leq 2L\gamma^2,
    \end{equation*}
    which together with \eqref{eq:Theta} and the choice of $\gamma$ in \eqref{eq:gamma-gen} further implies (note $1\leq t\leq t^2$ for $t\ge 1$)
    \begin{equation}\label{eq:rate-stoc-1}
        \begin{aligned}
            &\sigma^2\sum_{k=1}^{t-1}\frac{\Theta_{t-1}}{2\Theta_k}\Big(L \gamma^2+\theta_k\Big(\gamma+\frac{1-\beta_k}{\theta_{k+1}}\eta_k\Big)\Big)\leq \frac{\sigma^2L \gamma^2}{t(t+1)} \sum_{k=1}^{t-1}(k+1)(k+2)\\
            =&\frac{\sigma^2 L \gamma^2}{t(t+1)} \frac{(t-1)(t^2+4 t+6)}{3}\leq\frac{\sigma^2 L \gamma^2(t^2+4 t+6)}{3 t}\leq\frac{11t\sigma^2 L \gamma^2}{3}\leq\frac{11C\sigma}{3\sqrt{t}}.
        \end{aligned}
    \end{equation}
    Moreover, applying \eqref{eq:gamma-gen} in \eqref{eq:Phi_1} gives
    \begin{equation}\label{eq:Phi_1-stoc}
    \begin{aligned}
        \Phi_1=&f(x_1)-f^*+\frac{1}{L\min\{1/L^2,~C t^{-3/2}/\sigma L\}}\|x_1-x^*\|^2\\
        =&f(x_1)-f^*+\max\{L,\sigma t^{3/2}/C\}\|x_1-x^*\|^2\\
        \leq&f(x_1)-f^*+(L+\sigma t^{3/2}/C)\|x_1-x^*\|^2.
    \end{aligned}
    \end{equation}
    Finally, plugging \eqref{eq:Theta} with $k=t-1$, \eqref{eq:rate-stoc-1}, and \eqref{eq:Phi_1-stoc} into \eqref{eq:rate-general}, we obtain
    \begin{align*}
        \ebb[f(y_t)]-f^*\leq \frac{2(f(x_1)-f^*+L\|x_1-x^*\|^2)}{t(t+1)} + \frac{\sigma\|x_1-x^*\|^2}{C\sqrt{t}}+\frac{11C\sigma}{3\sqrt{t}},
    \end{align*}
    which proves (ii). The proof is now complete.
\end{proof}

\section{Convergence for Nonconvex Problems}
\label{sec:convergence-nc}
In this section, we move beyond convex objectives and turn to the nonconvex setting, where global optimality can no longer be guaranteed. Nevertheless, we show that \eqref{alg:g-sgdm} still admits meaningful convergence guarantees to stationary points.
For brevity, we focus on \eqref{alg:g-sgdm} with constant parameters $\beta\in[0,1)$, $\gamma\ge0$, and $\eta>0$. 

To facilitate our analysis, we consider the auxiliary sequence $\{w_k\}_{k\ge1}$ defined in \eqref{def:wk}, which satisfies the SGD-like property Lemma \ref{lem:w-sgd}. Additionally, from the definition \eqref{def:wk}, the second equation of \eqref{alg:g-sgdm}, and the initialization $m_0=0$, it is straightforward to verify that
\begin{equation}\label{eq:wk-xk}
    w_k=x_k-\frac{\beta\eta}{1-\beta} m_{k-1},~\forall k\ge 1.
\end{equation}
Before presenting the main result, we first state the following lemma, which provides an upper bound on $\ebb[\|m_k\|^2]$ and will be useful to our analysis.
\begin{lemma}\label{lem:m-var}
Let Assumption \ref{ass:s-bv} hold. Let $\{x_k\}_{k\ge 1}$ and $\{m_k\}_{k\ge 0}$ be the sequences generated by \eqref{alg:g-sgdm} with constant momentum parameter $\beta\in [0,1)$. Then for any $k\ge 0$, it holds that
\[
\ebb[\|m_k\|^2] \leq 2(1-\beta) \sum_{j=1}^k \beta^{k-j}\ebb[\|\nabla f(x_j)\|^2]+\frac{2(1-\beta)}{1+\beta} \sigma^2.
\]
\end{lemma}

\begin{proof}[Proof of Lemma \ref{lem:m-var}]
    The desired result holds trivially for $k=0$ since $m_0=0$. Hence we assume $k\ge 1$ in what follows. Applying the first equation in \eqref{alg:g-sgdm} recursively and noting that $m_0=0$, we obtain
    \[
    m_k = (1-\beta) \sum_{j=1}^k \beta^{k-j} g_j, 
    \]
    which together with Young's inequality further gives
    \begin{equation}\label{eq:m-bd-1}
    \begin{aligned}
        \|m_k\|^2 \leq& 2\Big\|(1-\beta) \sum_{j=1}^k \beta^{k-j} \nabla f(x_j)\Big\|^2+2\Big\|m_k-(1-\beta) \sum_{j=1}^k \beta^{k-j}\nabla f(x_j))\Big\|^2\\
        =&2(1-\beta)^2\Big\|\sum_{j=1}^k \beta^{k-j} \nabla f(x_j)\Big\|^2+2(1-\beta)^2 \Big\|\sum_{j=1}^k \beta^{k-j}(g_j-\nabla f(x_j))\Big\|^2.
    \end{aligned}
    \end{equation}
    To bound the first term on the RHS of \eqref{eq:m-bd-1}, we use the convexity of $\|\cdot\|^2$ to derive
    \begin{equation}\label{eq:m-bd-term1}
    \begin{aligned}
        \Big\|\sum_{j=1}^k \beta^{k-j} \nabla f(x_j)\Big\|^2=&\Big\|\Big(\sum_{j=1}^k \beta^{k-j}\Big)\sum_{j=1}^k \frac{\beta^{k-j}}{\sum_{j=1}^k \beta^{k-j}} \nabla f(x_j)\Big\|^2\\
        \leq&\Big(\sum_{j=1}^k \beta^{k-j}\Big)^2\sum_{j=1}^k \frac{\beta^{k-j}}{\sum_{j=1}^k \beta^{k-j}} \|\nabla f(x_j)\|^2\\
        \leq&\frac{1}{1-\beta}\sum_{j=1}^k \beta^{k-j} \|\nabla f(x_j)\|^2,
    \end{aligned}
    \end{equation}
    where the last inequality holds due to $\sum_{j=1}^k \beta^{k-j}=\sum_{j=0}^{k-1} \beta^j=\frac{1-\beta^k}{1-\beta}\leq \frac{1}{1-\beta}$. For the second term in \eqref{eq:m-bd-1}, we expand the square of norm to obtain
    \begin{equation}\label{eq:m-bd-2}
        \Big\|\sum_{j=1}^k \beta^{k-j}(g_j-\nabla f(x_j))\Big\|^2=\sum_{i=1}^k \sum_{j=1}^k\left\langle\beta^{k-i}\left(g_i-\nabla f(x_i)\right), \beta^{k-j}\left(g_j-\nabla f(x_j)\right)\right\rangle.
    \end{equation}
    For $i<j$, since $\ebb_i[g_i]=\nabla f(x_i)$ from Assumption \ref{ass:s-bv}, it follows that
    \begin{equation}\label{eq:m-bd-3}
        \ebb_j[\left\langle g_i-\nabla f(x_i), g_j-\nabla f(x_j)\right\rangle] = \left\langle g_i-\nabla f(x_i), \ebb_j[g_j-\nabla f(x_j)]\right\rangle = 0.
    \end{equation}
    Due to symmetry, we know that \eqref{eq:m-bd-3} also holds for $i>j$. Taking expectations on both sides of \eqref{eq:m-bd-2} and noting \eqref{eq:m-bd-3} for $i\neq j$, we obtain from Assumption \ref{ass:s-bv} that
    \begin{equation}\label{eq:m-bd-4}
        \ebb\Big[\Big\|\sum_{j=1}^k \beta^{k-j}(g_j-\nabla f(x_j))\Big\|^2\Big] = \sum_{j=1}^k\beta^{2(k-j)}\ebb[\|g_j-\nabla f(x_j)\|^2]\leq \sigma^2\sum_{j=1}^k\beta^{2(k-j)}\leq\frac{\sigma^2}{1-\beta^2},
    \end{equation}
    where the last inequality follows from $\sum_{j=1}^k\beta^{2(k-j)}=\sum_{j=0}^{k-1}\beta^{2j}=\frac{1-\beta^{2k}}{1-\beta^2}\leq\frac{1}{1-\beta^2}$. Therefore, taking expectations on both sides of \eqref{eq:m-bd-1} and plugging in \eqref{eq:m-bd-term1} and \eqref{eq:m-bd-4} leads to the desired result, which completes the proof.
\end{proof}

Leveraging Lemma \ref{lem:m-var}, we now derive the main convergence result for nonconvex objectives. The following theorem establishes that the iterates generated by \eqref{alg:g-sgdm} converge to stationary points at a sublinear rate in expectation.
\begin{theorem}\label{thm:nc}
    Let Assumption \ref{ass:s-bv} hold. Let $\{x_k\}_{k\ge 1}$ be the sequence generated by \eqref{alg:g-sgdm} with constant parameters $\beta\in [0,1)$, $\gamma\ge 0$, and $\eta>0$. If we choose
\[
0<\gamma+\eta \leq \frac{1-\beta}{3 L},
\]
then for any $t\ge 1$, it holds that
\[
    \frac{1}{t}\sum_{k=1}^t \ebb[\|\nabla f(x_k)\|^2]\leq \frac{2(f(x_1)-f^*)}{(\gamma+\eta)t}+\Big(\frac{\beta^2 L \eta}{1+\beta}+L(\gamma+\eta)\Big)\sigma^2.
\]
\end{theorem}

\begin{proof}[Proof of Theorem \ref{thm:nc}]
We first obtain from Lemma \ref{lem:w-sgd} and the $L$-smoothness of $f$ that
\begin{equation}\label{eq:nc-1}
\begin{aligned}
f(w_{k+1}) & \leq f(w_k)+\langle\nabla f(w_k), w_{k+1}-w_k\rangle+\frac{L}{2} \|w_{k+1}-w_k\|^2 \\
& =f(w_k)-(\gamma+\eta)\langle\nabla f(w_k), g_k\rangle+\frac{L (\gamma+\eta)^2}{2} \|g_k\|^2.
\end{aligned}
\end{equation}
From $\ebb_k[g_k]=\nabla f(x_k)$ and \eqref{eq:wk-xk}, we have for any $\rho>0$ that 
\begin{equation}\label{eq:nc-2}
\begin{aligned}
\ebb[\langle\nabla f(w_k), g_k\rangle]= & \ebb[\langle\nabla f(w_k)-\nabla f(x_k), \nabla f(x_k)\rangle]+\ebb[\|\nabla f(x_k)\|^2] \\
\ge & -\frac{\rho}{2} \ebb[\|\nabla f(w_k)-\nabla f(x_k)\|^2]- \frac{1}{2 \rho} \ebb[\|\nabla f(x_k)\|^2]+ \ebb[\|\nabla f(x_k)\|^2]\\
\ge & -\frac{\rho L^2}{2}\ebb[\|w_k-x_k\|^2]- \Big(\frac{1}{2 \rho}-1\Big) \ebb[\|\nabla f(x_k)\|^2]\\
= & -\frac{\rho L^2\beta^2\eta^2}{2(1-\beta)^2}\ebb[\|m_{k-1}\|^2]- \Big(\frac{1}{2 \rho}-1\Big) \ebb[\|\nabla f(x_k)\|^2].
\end{aligned}
\end{equation}
Taking expectations on both sides of \eqref{eq:nc-1} and applying \eqref{eq:nc-2} and Assumption \ref{ass:s-bv} further gives
\begin{equation}\label{eq:nc-3}
\begin{aligned}
\ebb[f(w_{k+1})] \leq & \ebb[f(w_k)]+\frac{\rho L^2\beta^2\eta^2(\gamma+\eta)}{2(1-\beta)^2}  \ebb[\|m_{k-1}\|^2] \\
& +\Big(\frac{1}{2 \rho}-1\Big)(\gamma+\eta) \ebb[\|\nabla f(x_k)\|^2]+\frac{L (\gamma+\eta)^2}{2} \ebb[\|g_k\|^2]\\
\leq & \ebb[f(w_k)]+\frac{\rho L^2\beta^2\eta^2(\gamma+\eta)}{2(1-\beta)^2} \ebb[\|m_{k-1}\|^2]+\frac{L(\gamma+\eta)^2}{2} \sigma^2 \\
& +\Big(\frac{1}{2 \rho}-1+\frac{L(\gamma+\eta)}{2}\Big)(\gamma+\eta) \ebb[\|\nabla f(x_k)\|^2].
\end{aligned}
\end{equation}
From Lemma \ref{lem:m-var}, we can bound $\ebb[\|m_{k-1}\|^2]$ via
\begin{align*}
   \ebb[\|m_{k-1}\|^2]\leq& 2(1-\beta) \sum_{j=1}^{k-1} \beta^{k-1-j}\ebb[\|\nabla f(x_j)\|^2]+\frac{2(1-\beta)}{1+\beta} \sigma^2\\
   \leq& \frac{2(1-\beta)}{\beta}\sum_{j=1}^k \beta^{k-j}\ebb[\|\nabla f(x_j)\|^2]+\frac{2(1-\beta)}{1+\beta} \sigma^2,
\end{align*}
applying which on \eqref{eq:nc-3} leads to
\begin{multline}\label{eq:nc-4}
\ebb[f(w_{k+1})] \leq \ebb[f(w_k)]+\frac{\rho L^2 \beta \eta^2 (\gamma+\eta)}{1-\beta} \sum_{j=1}^k \beta^{k-j} \ebb[\|\nabla f(x_j)\|^2]+\Big(\frac{\rho L^2 \beta^2 \eta^2(\gamma+\eta)}{1-\beta^2}+\frac{L(\gamma+\eta)^2}{2}\Big) \sigma^2 \\
+\Big(\frac{1}{2 \rho}-1+\frac{L(\gamma+\eta)}{2}\Big)(\gamma+\eta) \ebb[\|\nabla f(x_k)\|^2].
\end{multline}
Summing up \eqref{eq:nc-4} over $k=1,\dots, t$ and noting that $w_1=x_1$ further gives
\begin{equation}\label{eq:nc-5}
\begin{aligned}
\ebb[f(w_{t+1})] \leq & f(x_1)+\frac{\rho L^2 \beta \eta^2 (\gamma+\eta)}{1-\beta}\sum_{k=1}^t\sum_{j=1}^k \beta^{k-j} \ebb[\|\nabla f(x_j)\|^2]+\Big(\frac{\rho L^2 \beta^2\eta^2 (\gamma+\eta)}{1-\beta^2}+\frac{L(\gamma+\eta)^2}{2}\Big) t \sigma^2 \\
& +\Big(\frac{1}{2 \rho}-1+\frac{L(\gamma+\eta)}{2}\Big)(\gamma+\eta) \sum_{k=1}^t \ebb[\|\nabla f(x_k)\|^2]\\
\leq & f(x_1)+\Big(\frac{\rho L^2 \beta^2\eta^2 (\gamma+\eta)}{1-\beta^2}+\frac{L(\gamma+\eta)^2}{2}\Big) t \sigma^2 \\
& +\Big(\frac{1}{2 \rho}-1+\frac{L(\gamma+\eta)}{2}+\frac{\rho L^2 \beta \eta^2}{(1-\beta)^2}\Big)(\gamma+\eta) \sum_{k=1}^t \ebb[\|\nabla f(x_k)\|^2],
\end{aligned}
\end{equation}
where the second inequality holds due to
\[
\sum_{k=1}^t\sum_{j=1}^k \beta^{k-j} \ebb[\|\nabla f(x_j)\|^2]=\sum_{j=1}^t\Big(\sum_{k=j}^t \beta^{k-j}\Big) \ebb[\|\nabla f(x_j)\|^2]\leq \frac{1}{1-\beta}\sum_{j=1}^t \ebb[\|\nabla f(x_j)\|^2].
\]
Let $\rho=\frac{1-\beta}{2 L(\gamma+\eta)}$, then it follows from \eqref{eq:nc-5} and $f(w_{t+1})\ge f^*$ that
\begin{equation}\label{eq:nc-6}
\begin{aligned}
& \Big(1-\frac{L(\gamma+\eta)}{1-\beta}-\frac{L(\gamma+\eta)}{2}-\frac{\beta L \eta^2}{2(1-\beta)(\gamma+\eta)}\Big)(\gamma+\eta) \sum_{k=1}^t \ebb[\|\nabla f(x_k)\|^2] \\
\leq& f(x_1)-f^*+\Big(\frac{\beta^2 L \eta^2}{2(1+\beta)}+\frac{L(\gamma+\eta)^2}{2}\Big) t \sigma^2.
\end{aligned}
\end{equation}
From $0\leq\eta\leq\gamma+\eta$ and the condition $\gamma+\eta \leq \frac{1-\beta}{3 L}$, we obtain that
\begin{equation*}
\begin{aligned}
\frac{L(\gamma+\eta)}{1-\beta}+\frac{L(\gamma+\eta)}{2}+\frac{\beta L \eta^2}{2(1-\beta)(\gamma+\eta)} \leq& L(\gamma+\eta)\Big(\frac{1}{1-\beta}+\frac{1}{2}+\frac{\beta}{2(1-\beta)}\Big) \\
= & L(\gamma+\eta) \frac{3}{2(1-\beta)} \leq \frac{1}{2},
\end{aligned}
\end{equation*}
which together with \eqref{eq:nc-6} implies that
\begin{equation}\label{eq:nc-7}
    \begin{aligned}
        \frac{\gamma+\eta}{2}\sum_{k=1}^t \ebb[\|\nabla f(x_k)\|^2]\leq& \Big(1-\frac{L(\gamma+\eta)}{1-\beta}-\frac{L(\gamma+\eta)}{2}-\frac{\beta L \eta^2}{2(1-\beta)(\gamma+\eta)}\Big)(\gamma+\eta) \sum_{k=1}^t \ebb[\|\nabla f(x_k)\|^2] \\
        \leq& f(x_1)-f^*+\Big(\frac{\beta^2 L \eta^2}{2(1+\beta)}+\frac{L(\gamma+\eta)^2}{2}\Big) t \sigma^2.
    \end{aligned}
\end{equation}
Dividing both sides of \eqref{eq:nc-7} by $(\gamma+\eta)t/2$ and noting $0\leq\eta\leq\gamma+\eta$ further gives
\[
    \frac{1}{t}\sum_{k=1}^t \ebb[\|\nabla f(x_k)\|^2]\leq \frac{2(f(x_1)-f^*)}{(\gamma+\eta)t}+\Big(\frac{\beta^2 L \eta}{1+\beta}+L(\gamma+\eta)\Big)\sigma^2,
\]
which completes the proof.
\end{proof}

While Theorem \ref{thm:nc} establishes convergence to stationary points, it does not provide guarantees on the quality of these points, since stationary points in nonconvex problems may correspond to local minima or saddle points. To obtain stronger convergence guarantees, we consider the setting in which the objective function $f$ satisfies the $\mu$-Polyak-\L{}ojasiewicz (PL) condition, that is, 
\begin{equation}\label{eq:pl-condition}
\| \nabla f(x) \|^2 \ge 2 \mu(f(x)-f^*),~\forall x\in\rbb^d.
\end{equation}
Note that any $\mu$-strongly convex function satisfies the $\mu$-PL condition in \eqref{eq:pl-condition}. However, the PL condition does not imply convexity in general. For example, the function $f(x) = x^2+ 3 \sin^2 x$ satisfies the PL condition while being nonconvex. Moreover, the PL condition has been shown to hold for a broad class of deep linear and shallow neural networks \citep{charles2018stability, hardt2016identity, li2017convergence}.

In the upcoming theorem, we establish an improved convergence result for \eqref{alg:g-sgdm} under the PL condition, measured in terms of the objective gaps along the auxiliary points $\{w_k\}_{k\ge 1}$. To exploit the PL condition effectively, we introduce a family of Lyapunov functions $\{\varphi_k\}_{k\ge 1}$ that combine the function value gap with accumulated gradient norms:
\begin{equation}\label{def:varphi}
\varphi_k=f(w_k)-f^*+C \sum_{j=1}^{k-1} \beta^{k-1-j}\|\nabla f(x_j)\|^2,
\end{equation}
where $C>0$ is a constant to be specified in the proof. This choice of Lyapunov function enables us to establish a contraction property, up to error terms arising from stochastic noise, which plays a crucial role in the convergence analysis.
\begin{theorem}\label{thm:pl}
Let Assumption \ref{ass:s-bv} hold and suppose that $f$ satisfies the $\mu$-PL condition \eqref{eq:pl-condition}. Let $\{x_k\}_{k\ge 1}$ be the sequence generated by \eqref{alg:g-sgdm} with constant parameters $\beta\in [0,1)$, $\gamma\ge 0$, and $\eta>0$. If we choose
\begin{equation}\label{eq:pl-lr}
    0<\gamma+\eta \leq \min\Big\{\frac{1-\beta}{2 L},~\frac{1-\beta}{8 \beta^2 L}\Big\},
\end{equation}
then for any $t\ge 1$, it holds that
\[
\ebb[f(w_{t+1})-f^*]\leq\Big(1-\frac{\mu(\gamma+\eta)}{18}\Big)^{t}(f(x_1)-f^*)+\Big(\frac{3 \beta^2 \eta}{1+\beta}+\gamma+\eta\Big) \frac{9L}{\mu} \sigma^2,
\]
where $w_t$ is defined in \eqref{def:wk}.
\end{theorem}

\begin{proof}[Proof of Theorem \ref{thm:pl}]
For notational simplicity, denote
\[
    M:=\gamma+\eta-\frac{L(\gamma+\eta)^2}{1-\beta}-\frac{L(\gamma+\eta)^2}{2}.
\]
Consider the Lyapunov functions $\{\varphi_k\}_{k\ge 1}$ defined by \eqref{def:varphi} with
\begin{equation}\label{def:C}
C:=\frac{\frac{\beta^2 L \eta^2}{2}+\frac{M \mu(\mu+L) \beta^2 \eta^2}{1-\beta}}{1-\beta-\mu M+\frac{\mu(\mu+L) \beta^2 \eta^2}{1-\beta}}.
\end{equation}
We first establish some auxiliary inequalities regarding the introduced constants $C$ and $M$. To start off, it follows from $0<\gamma+\eta \leq \frac{1-\beta}{2 L}$ and $\beta\in[0,1)$ that
\begin{equation}\label{eq:pl-M>}
    M=(\gamma+\eta)\Big(1-\frac{(3-\beta) L}{2(1-\beta)}(\gamma+\eta)\Big) \ge \Big(1-\frac{3-\beta}{4}\Big)(\gamma+\eta) \ge \frac{1}{4}(\gamma+\eta)>0.
\end{equation}
On the other hand, we note from $\gamma+\eta\leq\frac{1-\beta}{2 L}$ and $\mu\leq L$ that
\begin{equation}\label{eq:pl-M<}
    M \leq \gamma+\eta \leq \frac{1-\beta}{2 L}\leq\frac{1-\beta}{2 \mu}.
\end{equation}
Hence we derive $C>0$ by noting from \eqref{eq:pl-M<} that
\begin{equation}\label{eq:pl-M-1}
1-\beta-\mu M \ge \frac{1-\beta}{2}\ge 0.
\end{equation}
Next, we show that $M-C$ is positive and bounded from both sides. To this end, first note from $0\leq\eta\leq\gamma+\eta \leq \frac{1-\beta}{8 \beta^2 L}$, \eqref{eq:pl-M-1}, and \eqref{eq:pl-M>} that
\begin{equation}\label{eq:pl-M-C-1}
\begin{aligned}
M(1-\beta-\mu M)-\frac{\beta^2 L\eta^2}{2} \ge& \frac{1-\beta}{2} M-\frac{\beta^2 L \eta^2}{2} \ge \frac{1-\beta}{8}(\gamma+\eta)-\frac{\beta^2 L}{2}(\gamma+\eta)^2 \\
= & (\gamma+\eta)\Big(\frac{1-\beta}{8}-\frac{\beta^2 L}{2}(\gamma+\eta)\Big) \ge \frac{1-\beta}{16}(\gamma+\eta).
\end{aligned}
\end{equation}
On the other hand, we obtain from $\mu\leq L$ and \eqref{eq:pl-lr} that
\begin{equation}\label{eq:pl-M-C-2}
\begin{aligned}
1-\beta-\mu M+\frac{\mu(\mu+L) \beta^2 \eta^2}{1-\beta} & \leq 1-\beta+\frac{\mu(\mu+L) \beta^2(\gamma+\eta)^2}{1-\beta} \\
& \leq 1-\beta+\frac{2L^2 \beta^2}{1-\beta}\frac{1-\beta}{2 L}\frac{1-\beta}{8 \beta^2 L} \\
& = 1-\beta+\frac{1-\beta}{8} = \frac{9}{8}(1-\beta).
\end{aligned}
\end{equation}
Then it follows from $C>0$, \eqref{eq:pl-M<}, \eqref{eq:pl-M-C-1}, and \eqref{eq:pl-M-C-2} that
\begin{equation}\label{eq:pl-M-C-3}
\frac{1-\beta}{2 \mu} \ge M \ge M-C=\frac{M(1-\beta-\mu M)-\frac{\beta^2 L \eta^2}{2}}{1-\beta-\mu M+\frac{\mu(\mu+L) \beta^2 \eta^2}{1-\beta}} \ge \frac{\frac{1-\beta}{16}(\gamma+\eta)}{\frac{9}{8}(1-\beta)}=\frac{\gamma+\eta}{18}>0.
\end{equation}

Next, we provide an upper bound on $\|\nabla f(x_k)\|^2$ using the $\mu$-PL condition \eqref{eq:pl-condition}. To this end, we note from the $L$-smoothness of $f$ and \eqref{eq:wk-xk} that
\begin{equation}\label{eq:pl-2}
\begin{aligned}
f(w_k) & \leq f(x_k)+\langle\nabla f(x_k), w_k-x_k\rangle+\frac{L}{2}\|w_k-x_k\|^2 \\
& =f(x_k)+\Big\langle\nabla f(x_k),-\frac{\beta \eta}{1-\beta} m_{k-1}\Big\rangle+\frac{L \beta^2 \eta^2}{2(1-\beta)^2}\|m_{k-1}\|^2 \\
& \leq f(x_k)+\frac{1}{2\mu}\|\nabla f(x_k)\|^2+\frac{\mu\beta^2\eta^2}{2 (1-\beta)^2}\|m_{k-1}\|^2+\frac{L \beta^2 \eta^2}{2(1-\beta)^2}\|m_{k-1}\|^2 \\
& =f(x_k)+\frac{1}{2\mu}\|\nabla f(x_k)\|^2+\frac{(\mu+L)\beta^2 \eta^2}{2(1-\beta)^2}\|m_{k-1}\|^2.
\end{aligned}
\end{equation}
Then it follows from the $\mu$-PL condition \eqref{eq:pl-condition} and \eqref{eq:pl-2} that
\[
\frac{1}{2 \mu}\|\nabla f(x_k)\|^2 \ge f(x_k)-f^* \ge f(w_k)-f^*-\frac{1}{2\mu}\|\nabla f(x_k)\|^2-\frac{(\mu+L)\beta^2 \eta^2}{2(1-\beta)^2}\|m_{k-1}\|^2,
\]
which further implies that
\begin{equation}\label{eq:pl-nabla_k}
\|\nabla f(x_k)\|^2 \ge \mu(f(w_k)-f^*)-\frac{\mu (\mu+L)\beta^2 \eta^2}{2(1-\beta)^2}\|m_{k-1}\|^2.
\end{equation}

Now we are ready to establish a recursive inequality for the Lyapunov functions $\{\varphi_k\}_{k\ge 1}$. Plugging $\rho=\frac{1-\beta}{2 L(\gamma+\eta)}$ into \eqref{eq:nc-3} gives
\[
\ebb[f(w_{k+1})] \leq \ebb[f(w_k)]+\frac{\beta^2 L \eta^2}{4(1-\beta)} \ebb[\|m_{k-1}\|^2]+\frac{L(\gamma+\eta)^2}{2} \sigma^2-M \ebb[\|\nabla f(x_k)\|^2],
\]
which together with the definition of $\{\varphi_k\}_{k\ge 1}$ \eqref{def:varphi} implies that
\begin{equation}\label{eq:pl-3}
\begin{aligned}
\ebb[\varphi_{k+1}]:=&\ebb\Big[f(w_{k+1})-f^*+C \sum_{j=1}^k \beta^{k-j}\|\nabla f(x_j)\|^2\Big]\\
=&\ebb[f(w_{k+1})-f^*]+C \ebb[\|\nabla f(x_k)\|^2]+\beta C \sum_{j=1}^{k-1} \beta^{k-1-j} \ebb[\|\nabla f(x_j)\|^2] \\
\leq & \ebb[f(w_k)-f^*]+\frac{\beta^2 L \eta^2}{4(1-\beta)} \ebb[\|m_{k-1}\|^2]+\frac{L(\gamma+\eta)^2}{2} \sigma^2-(M-C) \ebb[\|\nabla f(x_k)\|^2]\\
&+\beta C \sum_{j=1}^{k-1} \beta^{k-1-j} \ebb[\|\nabla f(x_j)\|^2].
\end{aligned}
\end{equation}
From $M-C>0$, applying \eqref{eq:pl-nabla_k} and Lemma \ref{lem:m-var} on \eqref{eq:pl-3}, we further obtain
\begin{equation}\label{eq:pl-4}
\begin{aligned}
\ebb[\varphi_{k+1}] \leq& [1-\mu(M-C)] \ebb[f(w_k)-f^*]+\Big(\frac{\beta^2 L \eta^2}{4(1-\beta)}+\frac{(M-C) \mu(\mu+L) \beta^2 \eta^2}{2(1-\beta)^2}\Big) \ebb[\|m_{k-1}\|^2] \\
&+\frac{L(\gamma+\eta)^2}{2} \sigma^2+\beta C \sum_{j=1}^{k-1} \beta^{k-1-j} \ebb[\|\nabla f(x_j)\|^2]\\
\leq & [1-\mu(M-C)] \ebb[f(w_k)-f^*]+\Big(\frac{\beta^2 L \eta^2}{2(1+\beta)}+\frac{(M-C) \mu(\mu+L) \beta^2 \eta^2}{1-\beta^2}+\frac{L(\gamma+ \eta)^2}{2}\Big) \sigma^2 \\
& +\Big(\beta C+\frac{\beta^2 L \eta^2}{2}+\frac{(M-C) \mu(\mu+L) \beta^2 \eta^2}{1-\beta}\Big) \sum_{j=1}^{k-1} \beta^{k-1-j} \ebb[\|\nabla f(x_j)\|^2].
\end{aligned}
\end{equation}
From the definition of $C$ \eqref{def:C}, it is straightforward to verify that
\begin{equation}\label{eq:pl-5}
\beta C+\frac{\beta^2 L \eta^2}{2}+\frac{(M-C) \mu(\mu+L) \beta^2 \eta^2}{1-\beta} = C(1-\mu M)\leq C[1-\mu(M-C)]
\end{equation}
Plugging \eqref{eq:pl-5} into \eqref{eq:pl-4} and noting \eqref{eq:pl-M-C-3} and $\mu\leq L$, we derive
\begin{equation}\label{eq:pl-6}
\begin{aligned}
\ebb[\varphi_{k+1}] & \leq[1-\mu(M-C)] \ebb[\varphi_k]+\Big(\frac{\beta^2 L \eta^2}{2(1+\beta)}+\frac{(M-C) \mu(\mu+L) \beta^2 \eta^2}{1-\beta^2}+\frac{L(\gamma+\eta)^2}{2}\Big) \sigma^2 \\
& \leq\Big(1-\frac{\mu(\gamma+\eta)}{18}\Big) \ebb[\varphi_k]+\Big(\frac{\beta^2 L \eta^2}{2(1+\beta)}+\frac{(\mu+L) \beta^2 \eta^2}{2(1+\beta)}+\frac{L(\gamma+\eta)^2}{2}\Big) \sigma^2 \\
& \leq\Big(1-\frac{\mu(\gamma+\eta)}{18}\Big) \ebb[\varphi_k]+\Big(\frac{3 \beta^2 L \eta^2}{2(1+\beta)}+\frac{L(\gamma+\eta)^2}{2}\Big) \sigma^2.
\end{aligned}
\end{equation}
Applying \eqref{eq:pl-6} recursively and noting that $\varphi_1=f(x_1)-f^*$ further leads to
\begin{equation*}
\begin{aligned}
&\ebb[f(w_{k+1})-f^*]\leq\ebb[\varphi_{k+1}]\\
\leq & \Big(1-\frac{\mu(\gamma+\eta)}{18}\Big)^k(f(x_1)-f^*)+\Big(\frac{3 \beta^2 L \eta^2}{2(1+\beta)}+\frac{L(\gamma+\eta)^2}{2}\Big) \sigma^2 \sum_{j=0}^{k-1}\Big(1-\frac{\mu(\gamma+\eta)}{18}\Big)^j \\
\leq & \Big(1-\frac{\mu(\gamma+\eta)}{18}\Big)^k(f(x_1)-f^*)+\Big(\frac{3 \beta^2 L \eta^2}{2(1+\beta)}+\frac{L(\gamma+\eta)^2}{2}\Big) \sigma^2 \cdot \frac{18}{\mu(\gamma+\eta)} \\
\leq & \Big(1-\frac{\mu(\gamma+\eta)}{18}\Big)^k(f(x_1)-f^*)+\Big(\frac{3 \beta^2 \eta}{1+\beta}+\gamma+\eta\Big) \frac{9L}{\mu} \sigma^2,
\end{aligned}
\end{equation*}
which completes the proof.
\end{proof}

\begin{remark}
    We compare Theorems \ref{thm:nc} and \ref{thm:pl} with the results established in \citep{liu2020improved}, which studied the convergence behavior of \eqref{alg:sgd-hb} with constant parameters $\beta\in[0,1)$ and $\eta>0$ in both nonconvex and strongly convex settings.  A key technical ingredient in their analysis is the construction of Lyapunov functions $\{\mathcal{L}_k\}_{k\ge 1}$ of the form
    \[
    \mathcal{L}_k=f(w_k)-f^*+\sum_{j=1}^{k-1} c_j\|x_{k+1-j}-x_{k-j}\|^2,
    \]
    where $w_k$ is defined by \eqref{def:wk} with $\gamma=0$, and $\{c_k\}_{k\ge 1}$ are carefully chosen positive constants.
    By appropriately specifying these constants, \citep{liu2020improved} derived a recursive inequality for the Lyapunov sequence $\{\mathcal{L}_k\}_{k\ge1}$. Specifically, for an $L$-smooth nonconvex objective $f$, they showed that if $\eta=O(\frac{1-\beta}{L})$, then \eqref{alg:sgd-hb} achieves the stationary convergence
    \begin{equation}\label{eq:liu-1}
        \frac{1}{t}\sum_{k=1}^t \ebb[\|\nabla f(x_k)\|^2]=O\left(\frac{f\left(x_1\right)-f^*}{t \eta}+L \eta \sigma^2\right).
    \end{equation}
    Moreover, under the additional assumption that $f$ is $\mu$-strongly convex, the same Lyapunov framework yields
    \begin{equation}\label{eq:liu-2}
        \ebb[f(w_k)-f^*]=O\Big((1-\eta \mu)^k+\frac{L}{\mu} \eta \sigma^2\Big).
    \end{equation}    
    In comparison, Theorems \ref{thm:nc} and \ref{thm:pl} improve upon the results of \citep{liu2020improved} in several respects. First, for general nonconvex $L$-smooth objectives, Theorem \ref{thm:nc} establishes convergence of \eqref{alg:g-sgdm} without introducing any Lyapunov function, thereby substantially simplifying the analysis. By setting $\gamma=0$ in Theorem \ref{thm:nc}, we recover the same order of convergence rate as in \eqref{eq:liu-1} for any step size $\eta\leq\frac{1-\beta}{3L}$. Furthermore, Theorem \ref{thm:pl} yields a convergence result comparable to \eqref{eq:liu-2} when $\gamma=0$. Importantly, this result does not require $\mu$-strong convexity; instead, it relies on the weaker PL condition, which is compatible with nonconvex settings. In addition, our analysis employs a different class of Lyapunov functions $\{\varphi_k\}_{k\ge1}$ defined in \eqref{def:varphi}. Finally, our results apply to the more general method \eqref{alg:g-sgdm}, thereby extending the convergence guarantees beyond \eqref{alg:sgd-hb}.
\end{remark}

\section{Experiments}
\label{sec:experiments}
In this section, we conduct numerical experiments to validate our theoretical results for \eqref{alg:g-sgdm} on both convex and nonconvex problems.

\subsection{Logistic Regression}
We begin by evaluating the performance of \eqref{alg:g-sgdm} on a well-studied convex problem: the binary logistic regression problem
\begin{equation}\label{prob:logistic}
\min_{x \in \rbb^d} f(x) = \frac{1}{n} \sum_{i=1}^n \log\big(1 + \exp(-b_i \langle x, a_i \rangle)\big).
\end{equation}
Here, $(a_i, b_i)\in \rbb^d\times \{0,1\}$ denotes the $i$-th sample, where $a_i\in\rbb^d$ is the feature vector and $b_i\in\{0,1\}$ is the corresponding label.
To conduct the experiments, we use the MNIST dataset~\citep{lecun2002gradient}, which consists of $n=60{,}000$ training samples with $d=784$ features. The data are preprocessed by normalizing each feature vector to have unit $\ell_2$-norm. In addition, we convert the dataset into a binary classification problem by assigning the first half of the digit classes to label $0$ and the second half to label $1$. The following methods are tested in this part:

\textbf{SGD:} Stochastic gradient descent with a constant step size $\gamma>0$, which corresponds to \eqref{alg:g-sgdm} with parameters $\beta_k\equiv 0$, $\gamma_k\equiv\gamma$, and $\eta_k\equiv 0$.

\textbf{HB-const:} The heavy-ball method with constant momentum parameter $\beta\in[0,1)$ and step size $\eta>0$. This method corresponds to \eqref{alg:g-sgdm} with parameters $\beta_k\equiv \beta$, $\gamma_k\equiv 0$, and $\eta_k\equiv \eta$.

\textbf{NAG-const:} Nesterov’s momentum method with constant parameters $\beta\in[0,1)$ and $\gamma>0$. This method corresponds to \eqref{alg:g-sgdm} with parameters $\beta_k\equiv \beta$, $\gamma_k\equiv \gamma$, and $\eta_k\equiv \beta\gamma/(1-\beta)$.

\textbf{NAG-varying:} The classical Nesterov accelerated gradient method with time-varying momentum parameters $\beta_k=\frac{k-1}{k+2}$ and constant step size $\gamma>0$. This method corresponds to \eqref{alg:sgd-nesterov} with parameters $\beta_k=\frac{k-1}{k+2}$, $\gamma_k\equiv \gamma$, and $\eta_k\equiv \eta>0$.

\textbf{SGDM-varying:} The generalized SGDM method \eqref{alg:g-sgdm} with time-varying parameters chosen according to Corollary~\ref{cor:cvx-vary}.

We first compare the above five methods in both full-batch and mini-batch settings:

\textbf{Full-batch (deterministic) setting.} Each method is run for $T=2000$ iterations using the full dataset (i.e., batch size is $n$). For HB-const, the step size is set to $\eta=1/L$. To ensure a fair comparison, we use the same step size $\gamma=1/L$ for SGD (which becomes GD), NAG-const, NAG-varying, and SGDM-varying.
Note that the step-size choices for GD, HB-const, and NAG-const satisfy the conditions of Theorem~\ref{thm:cvx-const-deter}.

\textbf{Mini-batch (stochastic) setting.} Each method is run with a batch size of $128$ for $200$ epochs (approximately $93{,}750$ total iterations). We set $\eta=0.1$ for HB-const and fix the step size $\gamma=0.1$ for SGD, NAG-const, NAG-varying, and SGDM-varying.

In both settings, we fix the momentum parameter to $\beta=0.9$ and use the same initial point, drawn randomly from a normal distribution, for all methods. The optimal solution of \eqref{prob:logistic} is computed using the Python package \textit{scipy}. We evaluate the performance of a point $x\in\rbb^d$ using the optimality gap $f(x)-f^*$. Figure~\ref{fig:log_full_mini} summarizes the results for all five methods in both settings. We observe that: 
\begin{itemize}
\item In the full-batch setting, GD, HB-const, and NAG-const generate iterates with monotonically decreasing optimality gaps, consistent with the guarantees of Theorem~\ref{thm:cvx-const-deter}. Moreover, SGDM-varying exhibits improved performance comparable to NAG-varying, validating the accelerated $O(1/t^2)$ convergence rate established in Corollary~\ref{cor:cvx-vary}(i).
\item In the mini-batch setting, SGD, HB-const, and NAG-const again produce decreasing optimality gaps, which aligns with Theorem~\ref{thm:cvx-const}. In contrast, SGDM-varying preserves accelerated convergence despite stochastic gradients, whereas NAG-varying fails to converge.
\end{itemize}

\begin{figure}[t]
    \centering
    \begin{subfigure}[b]{0.47\textwidth}
        \includegraphics[width=\textwidth]{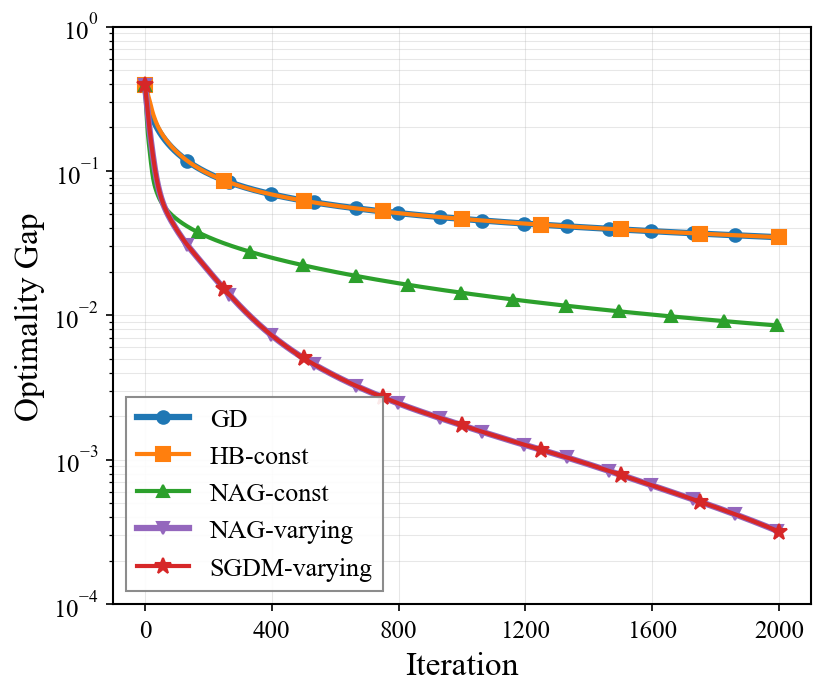}
        \caption{Full-batch}
        \label{fig:log_full}
    \end{subfigure}
    \hfill
    \begin{subfigure}[b]{0.48\textwidth}
        \includegraphics[width=\textwidth]{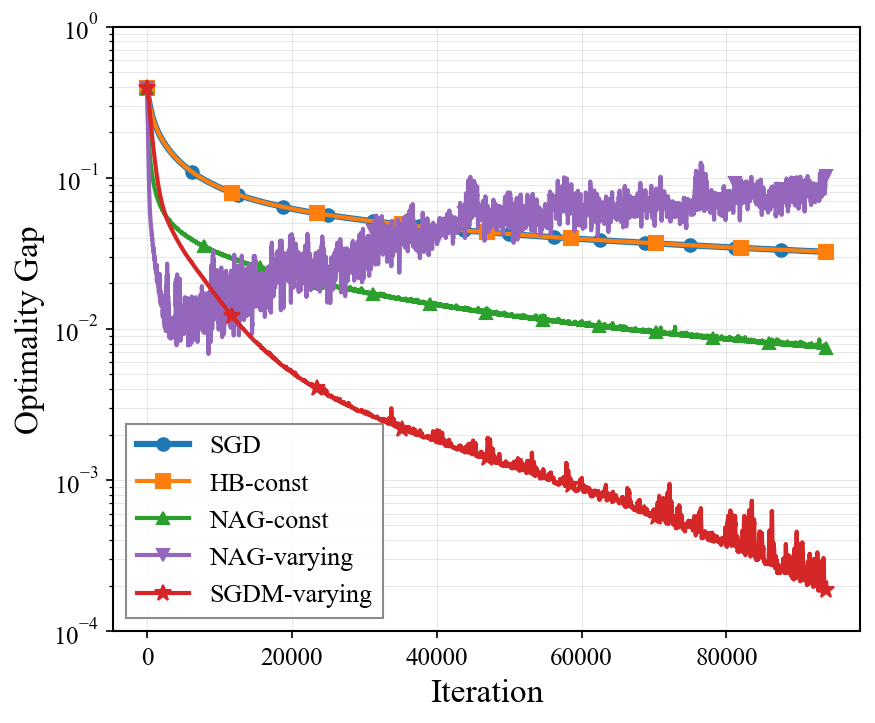}
        \caption{Mini-batch}
        \label{fig:log_mini}
    \end{subfigure}
    \caption{Evolution of optimality gap on the binary logistic regression problem in both full-batch (deterministic) and mini-batch (stochastic) settings.}
    \label{fig:log_full_mini}
\end{figure}

To further demonstrate the superior performance of SGDM-varying achieved through adaptive parameter adjustment, we conduct a detailed comparison between NAG-const and SGDM-varying. Specifically, we test both methods with different step sizes $\gamma \in \{0.01, 0.05, 0.2\}$. For a fair comparison, we fix the momentum parameter as $\beta = 0.9$ for NAG-const and $\beta_1 = 0.9$ for SGDM-varying. Both methods are run for $100$ epochs with a batch size of $128$.

Figure~\ref{fig:log_nag_comparison} shows the results for the different step sizes. We observe that both methods exhibit consistent improvement in the final training loss as $\gamma$ increases from $0.01$ to $0.2$. Furthermore, for the same $\gamma$, NAG-const decreases the optimality gap faster in the early stage of training, whereas SGDM-varying outperforms NAG-const in the later iterations. 

\begin{figure}
    \centering
    \includegraphics[width=0.98\linewidth]{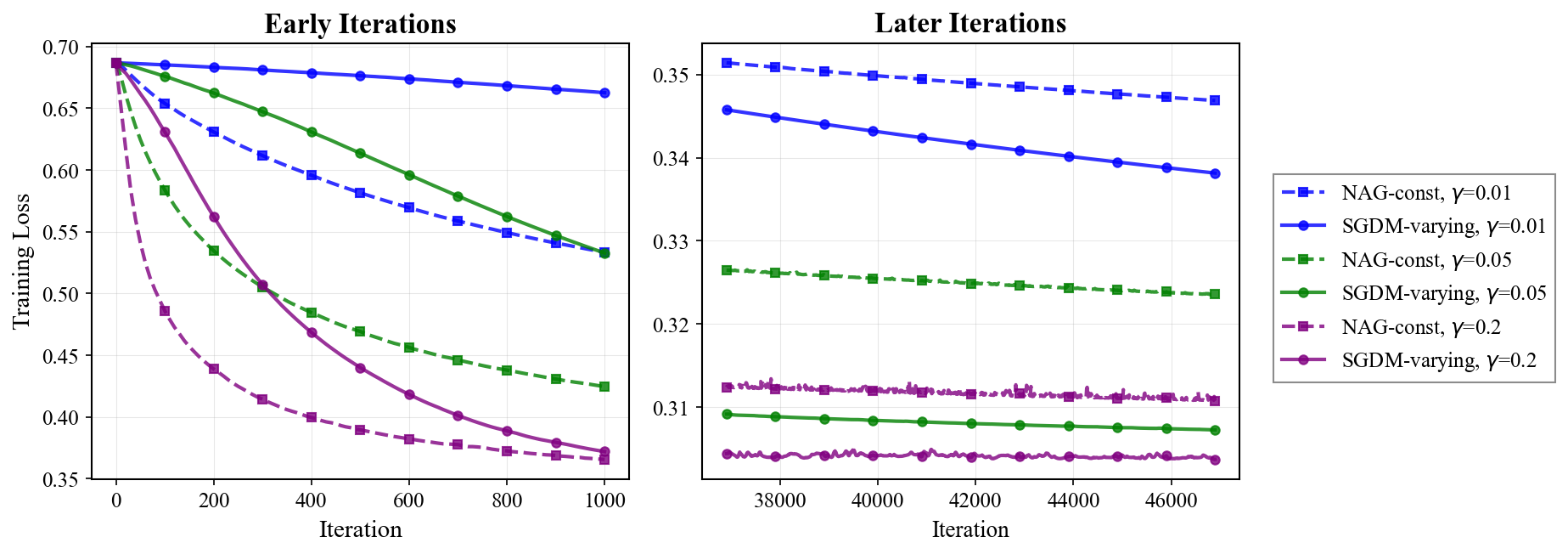}
    \caption{Comparison between NAG-const and SGDM-varying on the binary logistic regression problem with different step size $\gamma>0$. For better interpretability, we display the training loss curves in the early stage (first 1000 iterations) and the late stage (last 10000 iterations).}
    \label{fig:log_nag_comparison}
\end{figure}

\subsection{Image Classification}
In this part, we investigate the performance of \eqref{alg:g-sgdm} on nonconvex problems. Specifically, we train the ResNet-18 \citep{he2016deep} on the CIFAR-10 dataset \citep{krizhevsky2009learning}, which contains $60,000$ color images of size $32\times 32$ across 10 classes, with $50,000$ training samples and $10,000$ test samples. We adopt the standard cross-entropy loss for multi-class classification. All experiments are conducted in PyTorch on an NVIDIA A100 GPU.

Recall that in Theorem \ref{thm:nc}, we establish convergence for \eqref{alg:g-sgdm} with constant parameters $\{\beta, \gamma, \eta\}$. To validate this result, we test \eqref{alg:g-sgdm} under different combinations of step sizes $(\gamma, \eta)$ while fixing the momentum parameter $\beta = 0.9$. We consider several representative schemes, including special cases corresponding to classical optimization methods:
\begin{itemize}
    \item SGD: $\gamma=0.01$ and $\eta = 0$;
    \item Heavy-ball method: $\gamma=0$ and $\eta=0.01$;
    \item Nesterov's momentum method: $\gamma=0.01$ and $\eta = \frac{\beta \gamma}{1-\beta}=0.09$;
    \item Other generalized SGDM variants: $\gamma=0.01$, $\eta \in \{0.01, 0.05, 0.5\}$.
\end{itemize}
For each pair $(\gamma, \eta)$, we use a batch size of $128$ and train ResNet-18 for $100$ epochs on the training dataset. During training, we store the iterates along the optimization trajectories and evaluate the corresponding prediction accuracy on the test dataset.

Figure~\ref{fig:resnet18} shows the training loss and test accuracy curves for all considered methods. We observe that among them, plain SGD exhibits the poorest performance, showing the slowest decrease in training loss as well as the lowest prediction accuracy. While the heavy-ball method displays a training loss trajectory comparable to that of SGD, it achieves faster improvement in prediction accuracy, indicating superior generalization behavior. Nesterov’s momentum method further improves upon the heavy-ball method in both optimization efficiency and predictive performance. Finally, the step size combination $(\gamma,\eta) = (0.01, 0.5)$ consistently outperforms all other configurations, achieving the fastest convergence and the highest test accuracy.
These results demonstrate that \eqref{alg:g-sgdm} not only encompasses classical methods such as SGD, heavy-ball, and Nesterov acceleration, but also allows more flexible choices of $(\gamma, \eta)$ that can surpass these standard schemes. Properly balancing the step sizes $\gamma$ and $\eta$ can lead to both faster convergence and improved generalization in practice.
\begin{figure}
    \centering
    \includegraphics[width=0.9\linewidth]{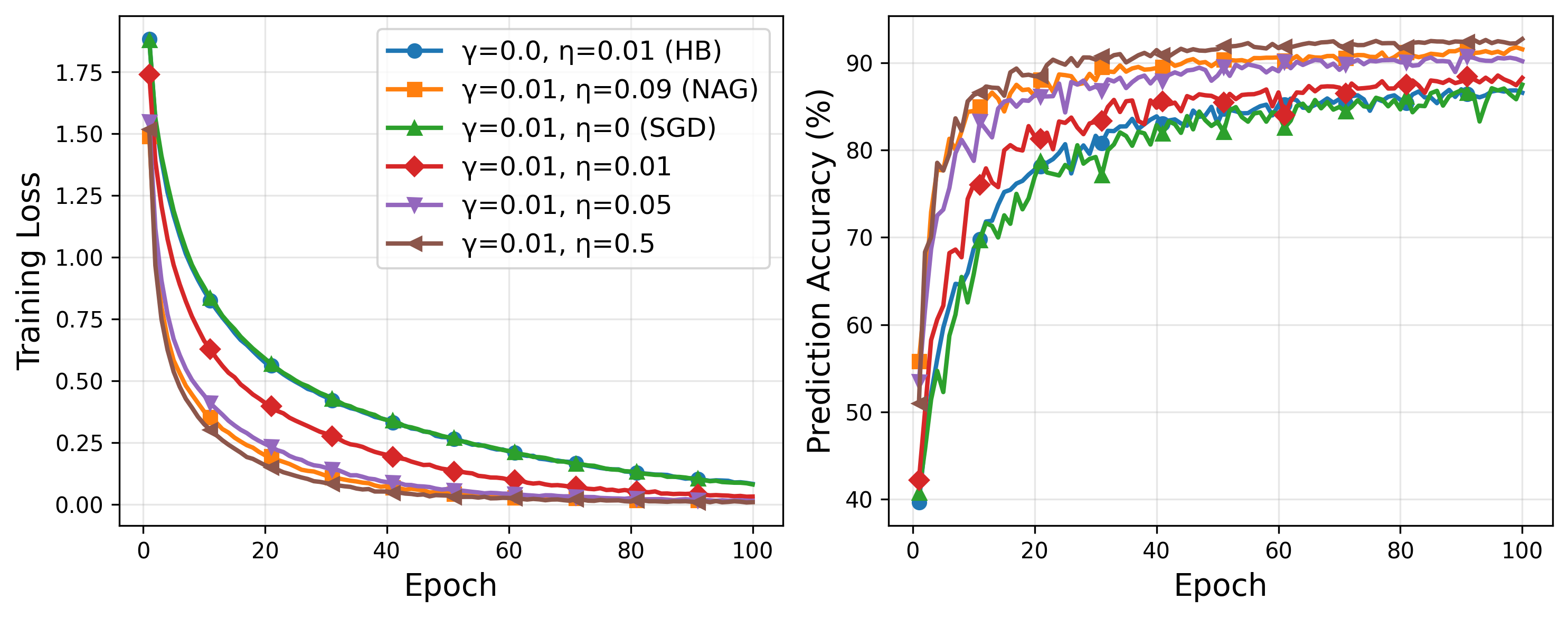}
    \caption{Training ResNet-18 on CIFAR-10 using \eqref{alg:g-sgdm} with different combinations of constant step sizes $(\gamma, \eta)$ with a fixed momentum parameter $\beta = 0.9$.}
    \label{fig:resnet18}
\end{figure}

\section{Conclusion}
\label{sec:conclusion}
In this paper, we propose a generalized stochastic gradient descent with momentum framework for smooth optimization, encompassing a broad class of momentum-based stochastic methods as special cases. Under standard smoothness and bounded variance assumptions, we provided a unified convergence analysis for both convex and nonconvex problems, offering new insights into the interaction between momentum and stochastic gradients.

For convex objectives, we established ergodic convergence results with flexible constant parameters and showed that, with appropriately designed time-varying parameters, the framework achieves an accelerated convergence rate of $O(1/t^2)$, akin to Nesterov's accelerated gradient method in deterministic settings, with an additional stochastic noise term. For nonconvex problems, we proved sublinear convergence to stationary points under constant parameters and demonstrated linear convergence to a neighborhood of the global optimum under the Polyak–\L{}ojasiewicz condition.

Importantly, our results rely only on mild smoothness and bounded variance assumptions and accommodate flexible parameter choices. This generality not only recovers existing convergence guarantees but also provides new ones for many popular momentum methods, including SGD with Polyak's or Nesterov's momentum, stochastic unified momentum (SUM), quasi-hyperbolic momentum (QHM), and momentum-added stochastic solver (MASS). Numerical experiments on logistic regression and ResNet-18 training validate our theoretical findings and highlight the practical advantages of the generalized framework.

While this work provides a solid theoretical foundation and empirical validation, several promising research directions remain open. Future work could extend this framework to adaptive step-size schemes, distributed or federated optimization, and nonsmooth or constrained problems. We believe that the generalized perspective and analytical tools developed in this work will serve as a foundation for further theoretical and algorithmic investigations in stochastic momentum-based optimization.

\section*{Acknowledgments}

This work was supported by the Wallenberg AI, Autonomous Systems and Software Program (WASP) funded by the Knut and Alice Wallenberg Foundation. The computations were enabled by the Berzelius resource provided by the Knut and Alice Wallenberg Foundation at the National Supercomputer Centre.

\setlength{\bibsep}{0.111cm}
\bibliographystyle{abbrvnat}
\bibliography{learning}
\end{document}